\documentclass[11pt]{article}

\usepackage{verbatim}
\usepackage{multicol}

\usepackage{amsthm}
\usepackage{amsfonts}
\usepackage{amsmath}
\usepackage{amssymb}
\usepackage{color}

\usepackage{pictex}
\usepackage{url,enumerate}

\begin{document}
	
\thispagestyle{empty}
	
\newtheorem{Lemma}{\bf LEMMA}[section]
\newtheorem{Theorem}[Lemma]{\bf THEOREM}
\newtheorem{Claim}[Lemma]{\bf CLAIM}
\newtheorem{Corollary}[Lemma]{\bf COROLLARY}
\newtheorem{Proposition}[Lemma]{\bf PROPOSITION}
\newtheorem{Example}[Lemma]{\bf EXAMPLE}
\newtheorem{Fact}[Lemma]{\bf FACT}
\newtheorem{definition}[Lemma]{\bf DEFINITION}
\newtheorem{Notation}[Lemma]{\bf NOTATION}
\newtheorem{remark}[Lemma]{\bf REMARK}

\title{Symmetric implication zroupoids and identities of Bol-Moufang type}

\author{Juan M. Cornejo and
Hanamantagouda P. Sankappanavar}



\date{}

\maketitle

{\color{red} This paper has been pubished by Soft Comput onlline on 25 October 2017.
DOI 10.1007/s00500-017-2869-z}

\begin{abstract}
	 An algebra $\mathbf A = \langle A, \to, 0 \rangle$, where $\to$ is binary and $0$ is a constant, is called an {\it implication zroupoid} ($\mathcal I$-zroupoid, for short) if $\mathbf A$ satisfies the identities: (I):  $(x \to y) \to z \approx ((z' \to x) \to (y \to z)')'$,  and  (I$_{0}$):  $ 0'' \approx 0$, where $x' : = x \to 0$.  An implication zroupoid is {\it symmetric} if it satisfies the identities: $x'' \approx x$ and $(x \to y')' \approx (y \to x')'$.  An identity is of {\it Bol-Moufang type} if it contains only one binary operation symbol, one of its three variables occurs twice on each side, each of the other two variables occurs once on each side, and the variables occur in the same (alphabetical) order on both sides of the identity. 
	 
	  In this paper we make a systematic analysis of all $ 60$  identities of Bol-Moufang type in the variety $\mathcal S$ of  symmetric $\mathcal I$-zroupoids.  
We show that $47$ of the subvarieties of $\mathcal S$, defined by the identities of Bol-Moufang type are equal to the variety $\mathcal{SL}$ of 
$\lor$-semilattices with the least element $0$ and, one of them is equal to $\mathcal S$.  Of the remaining 12, there are only $3$ distinct ones.  We also give an explicit description of the poset of the (distinct) subvarieties of $\mathcal S$ of Bol-Moufang type. 
\end{abstract}

\thispagestyle{empty}

\section{Introduction}

In 1934, Bernstein gave a system of axioms for Boolean algebras in \cite{bernstein1934postulates} using implication alone.  
Even though his system was not equational, it is not hard to see that one could easily convert it into an equational one by using an additional constant.  
In 2012, the second author extended this ``modified Bernstein's theorem''
to De Morgan algebras in \cite{sankappanavarMorgan2012}.  
Indeed, he shows in \cite{sankappanavarMorgan2012} that the varieties of De Morgan algebras, Kleene algebras, and  Boolean algebras are term-equivalent, respectively, to the varieties, $\mathcal{DM}$, $\mathcal{KL}$, and $\mathcal{BA}$ (defined below)
whose defining axioms use only an implication $\to$ and a constant $0$.    
The primary role played by the identity (I): $(x \to y) \to z \approx [(z' \to x) \to (y \to z)']'$, where $x' : = x \to 0$, in the axiomatization of each of those new varieties motivated the second author to introduce a new equational class of algebras called ``Implication zroupoids'' in   
\cite{sankappanavarMorgan2012}.  It also turns out that this new variety contains the variety $\mathcal{SL}$ (defined below) which is shown (see \cite{cornejo2015implication}) to be term-equivalent to the variety of $\lor$-semilattices with the least element $0$.

\begin{definition}
	An algebra $\mathbf A = \langle A, \to, 0 \rangle$, where $\to$ is binary and $0$ is a constant, is called a {\rm zroupoid}.
	A zroupoid $\mathbf A = \langle A, \to, 0 \rangle$ is an {\rm Implication zroupoid} \rm ($\mathcal I$-zroupoid, for short\rm) if $\mathbf A$ satisfies:
	\begin{itemize}
		\item[\rm{(I)}] 	$(x \to y) \to z \approx ((z' \to x) \to (y \to z)')'$, where $x' : = x \to 0$,
		\item[{\rm (I$_{0}$)}]  $ 0'' \approx 0$.
	\end{itemize}
$\mathcal{I}$ denotes the variety of implication zroupoids.
The varieties $\mathcal{DM}$, 
$\mathcal I_{1,0}$, $\mathcal C$ and $\mathcal{SL}$ 
are defined relative to $\mathcal{I}$, respectively, by the following identities: 
$$
\begin{array}{ll}
{\rm (DM)} & (x \to y) \to x \approx x \mbox{ (De Morgan Algebras)}; \\
{\rm (I_{1,0})} & x' \approx x; \\
{\rm (C)} & x \to y \approx y \to x; \\
{\rm (SL)}  & x' \approx x \mbox{ and } x \to y \approx y \to x.
\end{array}
$$
The varieties 
$\mathcal{KL}$ and $\mathcal{BA}$ 
are defined relative to $\mathcal{DM}$, respectively, by the following identities: 
$$
\begin{array}{ll}
{\rm (KL)} & (x \to x) \to (y \to y) \approx y \to y  \mbox{ (Kleene algebras)} ; \\
{\rm (BA)} & x \to x \approx 0'  \mbox{ (Boolean algebras)}.\\
\end{array}
$$ 
\end{definition}

As proved in \cite{sankappanavarMorgan2012}, the variety $\mathcal{I}$  
generalizes the variety of De Morgan algebras and exhibits several interesting properties; for example, 
the identity $x''' \to y \approx x' \to y$ holds in $\mathcal{I}$.  
Several new subvarieties of $\mathcal{I}$ are also introduced and investigated in
\cite{sankappanavarMorgan2012}.  The (still largely unexplored) lattice of subvarieties of $\mathcal{I}$ seems to be fairly complex.  In fact, Problem 6 of \cite{sankappanavarMorgan2012} calls for an investigation of the structure of the lattice of subvarieties of $\mathcal{I}$.

The papers  \cite{cornejo2015implication}, \cite{cornejo2016order} and \cite{cornejo2016semisimple} have addressed further, but still partially, the above-mentioned problem by introducing new subvarieties of $\mathcal{I}$ and investigating relationships among them.  The (currently known) size of the poset of subvarieties of $\mathcal{I}$ is at least 24; but it is still unknown whether the lattice of subvarieties is finite or infinite. Two of the subvarieties of $\mathcal{I}$ are:
 $\mathcal{I}_{2,0}$ and $\mathcal{MC}$ which are defined relative to $\mathcal{I}$, respectively, by the following identities, where $x \land y := (x \to y')' $ :\\      
\begin{equation} \label{eq_I20}  \tag{{\rm I}$_{2,0}$}
x'' \approx x;
\end{equation}
\begin{equation} \label{eq_MC} \tag{MC}
x \wedge y \approx y \wedge x.
\end{equation} 

For a somewhat more detailed summary of the results contained in the above-mentioned papers, we refer the reader to the introduction of \cite{cornejo2016derived}.

\begin{definition}
	Let $\mathbf{A} \in \mathcal{I}$.  $\mathbf{A}$ is {\rm involutive} if $A \in \mathcal{I}_{2,0}$.  $\mathbf{A}$ is {\rm meet-commutative} if $A \in \mathcal{MC}$.  $\mathbf{A}$ is {\rm symmetric} if $\mathbf{A}$ is both involutive and meet-commutative.  Let $\mathcal{S}$ denote the variety of symmetric $\mathcal{I}$-zroupoids.  In other words, $\mathcal{S} = \mathcal{I}_{2,0} \cap \mathcal{MC}$.
\end{definition}

In the present paper we are interested in the subvarieties of $\mathcal{S}$ defined by certain {\it weak associative laws}, called ``Bol-Moufang'' laws. 
A precise definition of a {\it weak associative law} appears in \cite{Ku96}, which is essentially restated below.

\begin{definition}
	Let $n \in \mathbb{N}$ and let $\mathcal{L} := \langle \times \rangle$, where $\times$ is a binary operation symbol.  A (groupoid) term in $\mathcal{L}$ 
is of length $n$ if the number of occurrences of variables (not necessarily distinct) is $n$.  A weak associative law of length $n$ in $\mathcal{L}$ is an identity of the form $p \approx q$, where $p$ and $q$ are of length $n$ and contain the same variables that occur in the same order in both $p$ and $q$ \rm(only the bracketing is possibly different\rm).     
\end{definition}
We note that the identities of {\it Bol-Moufang type} investigated in \cite{Fe69} are weak associative laws of length 4 in three distinct variables, with one of them repeated (see below for a more precise definition).

The following (general) problem presents itself naturally.\\

\noindent PROBLEM:  Let $\mathcal V$ be a given variety of algebras (whose language includes a  binary operation symbol, say, $\times$).  
Investigate the subvarieties of $\mathcal V$ defined by   
weak associative laws {\rm (}with respect to $\times${\rm)} 
and their mutual relationships.\\ 

Special cases of the above problem have already 
been considered in the literature, wherein the weak associative laws are the identities of Bol-Moufang type, and the variety ${\mathcal V}$ is the variety of quasigroups or the variety of loops
(see \cite{Fe69}, \cite{Ku96}, \cite{PV05a}, \cite{PV05b}).  

\cite{Fe69} has noted that there are $60$ weak associative laws of length 4 in three distinct variables (with one variable repeated).  Since the Moufang laws and the Bol identities were among those $60$ laws, the author of that paper called  such identities as those of Bol-Moufang type.  For more information about these identities in the context of quasigroups and loops, see \cite{Ku96, PV05a, PV05b}.

In this paper and in \cite{cornejo2016weak associative}, we carry out a complete analysis of weak associative identities of lenth $\leq 4$, in the context of the variety ${\mathcal S}$ of symmetric implication zroupoids.     
We focus, in this paper, exclusively on the systematic study of the identities of Bol-Moufang type in the context of symmetric implication zroupoids, while \cite{cornejo2016weak associative} will focus on the rest of the weak associative identities.

Without loss of generality, we will assume that the variables in $p$ and $q$ occur  
alphabetically.  
So, we can say (more explicitly) that
an identity $p \approx q$, in the language $\langle \to \rangle$, is of \emph{Bol-Moufang type} if:

\begin{enumerate}
	\item[(i)] 
	$p$ and $q$ contain the same three variables (chosen alphabetically),  
	\item[(ii)]  the length of $p$ = $4$ = the length of $q$, 
	\item[(iii)] the order in which the variables appear in $p$ is exactly the same as the order in which they appear in $q$.
\end{enumerate}

Let $B$ denote the set of identities of Bol-Moufang type.

The systematic notation for the identities in $B$, presented below, was developed in \cite{PV05a}.

Let $x, y, z$ be all the variables appearing in the identities of $B$.  Without loss
of generality, we can assume that they appear in the terms in the alphabetical order.
Then, the 6 ways in which the 3 variables can form a word of length 4 are shown 
below:\\
\begin{align*}
&\rm A:  \quad xxyz;   \qquad   \quad  &\rm B: \quad  xyxz; \\
&\rm C: \quad xyyz;   \qquad   \quad   &\rm D: \quad xyzx; \\
&\rm E:  \quad xyzy;  \qquad  \quad   &\rm F:  \quad xyzz.\\
\end{align*}
It is clear that there are exactly 5 ways in which a word of size 4 can be bracketed, as given below:\\
1:  \quad $o \to (o \to (o \to o))$; \\
2:  \quad $o \to ((o \to o) \to o)$; \\
3:  \quad $(o \to o) \to (o \to o)$;\\
4:  \quad $(o \to (o \to o)) \to o$; \\
5:  \quad $((o \to o) \to o) \to o$.\\

Let ($X_{ij}$) with $X \in \{\rm A, \dots, \rm F\}$, $1 \leq i < j \leq 5$, denote the identity from $B$ 
whose left-hand side is bracketed according to $i$, and
whose right-hand side is bracketed according to $j$.  For instance, ($A_{12}$) is the identity
$x \to (x \to (y \to z))  \approx x \to ((x \to y) \to z)$.  We note that B$_{15}$ is a Moufang identity, while E$_{25}$ is a Bol identity.

It is noted in \cite{Fe69}  that there are ($6 \times (4 + 3 + 2 + 1) =) \ 60$ nontrivial identities in B.

We will denote by $\mathcal{X}_{ij}$ the subvariety of $\mathcal S$ defined
by the identity ($X_{ij}$).

We will show in Section 3 that $47$ of these $60$ varieties coincide with $\mathcal{SL}$ (and hence with each other).  In Section 4, we prove our main result that there are 4 nontrivial varieties of Bol-Moufang type that are distinct from each other.  Furthermore, we describe explicitly the poset formed by them, together with the variety $\mathcal{BA}$ (which is contained in some of them). 

We would like to acknowledge that the software ``Prover 9/Mace 4'' developed by McCune \cite{Mc} has been useful to us in some of our findings presented in this paper.  We have used it to find examples and to check some conjectures.

\section{Preliminaries}

We refer the reader to the standard references
\cite {balbesDistributive1974}, \cite{burrisCourse1981} and \cite{Ra74} for concepts and results used, but not explained, in this paper.

Recall from \cite{sankappanavarMorgan2012} that $\mathcal{SL}$ is the variety of semilattices with a least element $0$. It was shown in \cite{cornejo2015implication} that $\mathcal{SL}   
= \mathcal{C} \cap \mathcal{I}_{1,0} $.

The two-element algebras $\mathbf{2_s}$, $\mathbf{2_b}$ were introduced in \cite{sankappanavarMorgan2012}.  Their  
operations $\to$ are respectively as follows:  \\

\begin{minipage}{0.3 \textwidth}
	\begin{tabular}{r|rr}
		$\to$: & 0 & 1\\
		\hline
		0 & 0 & 1 \\
		1 & 1 & 1
	\end{tabular} 
	
\end{minipage}
\begin{minipage}{0.3 \textwidth}
	\begin{tabular}{r|rr}
		$\to$: & 0 & 1\\
		\hline
		0 & 1 & 1 \\
		1 & 0 & 1
	\end{tabular}   
\end{minipage}           
\ \\ \ \\
Notice that $\mathcal{V}(\mathbf{2_b}) = \mathcal{BA}$.
Recall also from \cite[Corollary 11.4]{cornejo2015implication} \label{CorSL} that
$\mathcal{V}(\mathbf{2_s}) = \mathcal{SL}$.  The following lemma easily follows from the definition of $\land$ given earlier in the Introduction.

\begin{Lemma} \label{lemma_SL_I10_C}
	$\mathcal{MC} \cap \mathcal{I}_{1,0} \subseteq  \mathcal{C} \cap \mathcal{I}_{1,0} = \mathcal{SL}$.  
\end{Lemma}

\begin{Lemma} {\bf \cite[Theorem 8.15]{sankappanavarMorgan2012}} \label{general_properties_equiv}
	Let $\mathbf A$ be an  $\mathcal{I}$-zroupoid. Then the following are equivalent:
	\begin{enumerate}[{\rm (a)}]
		\item $0' \to x \approx x$, \label{TXX} 
		\item $x'' \approx x$,
		\item $(x \to x')' \approx x$, \label{reflexivity}
		\item $x' \to x \approx x$. \label{LeftImplicationwithtilde}
	\end{enumerate}
\end{Lemma}

\begin{Lemma}{\bf \cite{sankappanavarMorgan2012}} \label{general_properties}
	Let $\mathbf A \in  \mathcal{I}_{2,0}$. Then  
	\begin{enumerate}[{\rm (a)}]
		\item $x' \to 0' \approx 0 \to x$, \label{cuasiConmutativeOfImplic2}
		\item $0 \to x' \approx x \to 0'$. \label{cuasiConmutativeOfImplic}
	\end{enumerate}
\end{Lemma}

\begin{Lemma} \label{general_properties3}  \label{general_properties2}
	Let $\mathbf A \in \mathcal{I}_{2,0}$. Then $\mathbf A$ satisfies:
	\begin{enumerate}[{\rm (a)}]
		\item $(x \to 0') \to y \approx (x \to y') \to y$, \label{281014_05}
		\item $x \to (0 \to x)' \approx x'$, \label{291014_02}
		\item $(y \to x) \to y \approx (0 \to x) \to y$, \label{291014_10}
		\item $(0 \to x) \to (0 \to y) \approx x \to (0 \to y)$, \label{311014_06}
		\item $x \to y \approx x \to (x \to y)$, \label{031114_04} 
		\item $0 \to (0 \to x)' \approx 0 \to x'$, \label{031114_07}
		\item $0 \to (x' \to y)' \approx x \to (0 \to y')$, \label{071114_01}
		\item $0 \to (x \to y) \approx x \to (0 \to y)$, \label{071114_04}
		\item $0 \to (x \to y')' \approx 0 \to (x' \to y)$, \label{191114_05}
		\item $x \to (y \to x') \approx y \to x'$, \label{281114_01}		
		\item $(x \to y')' \to z \approx x \to (y \to z)$. \label{140715_20} 
	\end{enumerate}
\end{Lemma}

\begin{proof}

	For the proofs of items (\ref{281014_05}), (\ref{291014_02}),  
	(\ref{291014_10}), 
	(\ref{031114_07}), (\ref{071114_01}),  (\ref{071114_04}), 
	(\ref{191114_05})  
	and (\ref{281114_01}) 
	we refer the reader to \cite[Lemma 2.7]{cornejo2015implication}.  
	The proofs of items 
	(\ref{311014_06}) and (\ref{031114_04})
	are given in \cite[Lemma 2.6]{cornejo2016order}, 
	while the proof of the item (\ref{140715_20}) can be found in \cite[Lemma 3.3]{cornejo2016derived}.	
\end{proof}

\begin{Lemma}  \label{lemma_070616_01}
	Let $\mathbf A \in \mathcal{I}_{2,0}$ such that $\mathbf A \models 0 \to x \approx x$, then $\mathbf A \models (x \to y)' \approx x' \to y'$.
\end{Lemma}

\begin{proof}
	Let $a,b \in A$. Hence, we have that
	\[
	\begin{array}{lcll}
	a' \to b'	& = & 0 \to (a' \to b') & \mbox{by hyphotesis} \\
	&=& a' \to (0 \to b') & \mbox{by 
		Lemma \ref{general_properties2} (\ref{071114_04})}\\
	& = &  0 \to (a'' \to b)' & \mbox{by Lemma \ref{general_properties3} (\ref{071114_01})}\\ 	
	& = &  0 \to (a \to b)' &  \mbox{by (\ref{eq_I20})} \\ 
	& = & (a \to b)'  & \mbox{by hyphotesis}.
	\end{array}
	\]
	This completes the proof.
\end{proof}

\vspace{1cm}
\section{Symmetric $\mathcal{I}$-zroupoids of Bol-Moufang type}

Recall that the variety $\mathcal{S} = \mathcal{I}_{2,0} \cap \mathcal{MC}$,  
which was investigated in \cite{cornejo2015implication}.  
Throughout this section, $\mathbf{A} \in \mathcal{S}$.

In this section our goal is to prove that 47 of the subvarieties of $\mathcal{S}$ of Bol-Moufang type are equal to $\mathcal{SL}$ and hence are equal to each other.  First, we present some new properties of $\mathcal{S}$ which will be useful later in this paper.

\begin{Lemma} \label{properties_of_I20_MC}
	 
	$\mathbf A$ satisfies:
	\begin{enumerate}[{\rm (a)}]
		\item $x \to (y \to z) \approx y \to (x \to z)$,  \label{310516_01}
		\item $x' \to y \approx y' \to x$,  \label{310516_02}
		\item  $x \to y \approx y' \to x'$, \label{021017_01}
		\item  $x \to y' \approx y \to x'$.  \label{021017_02}
	\end{enumerate}
\end{Lemma}

\begin{proof} Let $a,b,c \in A$.  Then
	\begin{enumerate}[{\rm (a)}]
		\item 
		$$
		\begin{array}{lcll}
		a \to (b \to c)	& = & (a \to b')' \to c & \mbox{by Lemma \ref{general_properties3} (\ref{140715_20})} \\
		& = & (b \to a')' \to c & \mbox{by (\ref{eq_MC})} \\
		& = & b \to (a \to c) & \mbox{by Lemma \ref{general_properties3} (\ref{140715_20})}. 
		\end{array}
		$$
		\item 
		$$
		\begin{array}{lcll}
		a' \to b	& = & (a' \to b'')'' & \mbox{by (\ref{eq_I20})} \\
		& = & (b' \to a'')'' & \mbox{by (\ref{eq_MC})} \\
		& = & b' \to a & \mbox{by (\ref{eq_I20}).} 
		\end{array}
		$$
\item  
$$
\begin{array}{lcll}
a \to b & = & a'' \to b & \mbox{by (\ref{eq_I20})} \\
& = & b' \to a' & \mbox{by item (\ref{310516_02})}
\end{array}
$$

\item 
$$
\begin{array}{lcll}
a \to b' & = & a'' \to b' & \mbox{by (\ref{eq_I20})} \\
& = & b'' \to a' & \mbox{by item (\ref{310516_02})} \\
& = & b \to a' & \mbox{by (\ref{eq_I20})} 
\end{array}
$$
\end{enumerate}
\end{proof}

\begin{Lemma} \label{310516_09}
	Let $\mathbf A \models x \to x \approx x$. Then $\mathbf A \models x' \approx x$. 
\end{Lemma}

\begin{proof}
	Let $a \in A$. Then
	$$
	\begin{array}{lcll}
	a' & = & (a \to a)' & \mbox{by hypothesis} \\
	& = & (a \to a'')' & \mbox{by (\ref{eq_I20})} \\
	& = & (a' \to a')' & \mbox{by (\ref{eq_MC})} \\
	& = & a'' & \mbox{by hypothesis} \\
	& = & a & \mbox{by (\ref{eq_I20})}.
	\end{array}
	$$
\end{proof}

\begin{Lemma} \label{010616_01}
	
	Let $\mathbf A \models x' \approx x \to x$ and $\mathbf A \models 0' \approx 0$. Then $\mathbf A \models x' \approx x$.
\end{Lemma}

\begin{proof}
	Let $a \in A$. Then
	$$
	\begin{array}{lcll}
	a & = & a'' & \mbox{} \\
	& = & (a \to a)' & \mbox{by hypothesis} \\
	& = & (a \to a) \to 0 & \mbox{} \\
	& = & (a \to a) \to 0'  & \mbox{by hypothesis} \\
	& = & \{(0'' \to a) \to (a \to 0')'\}' & \mbox{by (I)} \\
	& = & \{(0 \to a) \to (a \to 0')'\}' & \mbox{by (\ref{eq_I20})} \\
	& = & \{(0 \to a) \to (a \to 0)'\}' & \mbox{by hypothesis} \\
	& = & \{(0 \to a) \to a''\}' & \mbox{} \\
	& = & \{(0 \to a) \to a\}' & \mbox{by (\ref{eq_I20})} \\
	& = & \{(a \to a) \to a\}' & \mbox{by Lemma \ref{general_properties3} (\ref{291014_10})} \\
	& = & (a' \to a)' & \mbox{by hypothesis} \\
	& = & a' & \mbox{by Lemma \ref{general_properties_equiv} (\ref{LeftImplicationwithtilde})},
	\end{array}
	$$
	completing the proof.
\end{proof}

\begin{Lemma} \label{040716_03}
	Let 	 
	$\mathbf A \models 0 \to (x \to x) \approx x \to x.$ Then \\
	$\mathbf A \models (x \to x) \to (y \to z)  \approx ((x \to x) \to y) \to z.$
\end{Lemma}

\begin{proof}
	Let $a,b,c \in A$. Now,
	\[
	\begin{array}{lcll}
	(a \to a) \to b' & = & (0 \to (a \to a)) \to b' & \mbox{by hypothesis} \\
	& = & (b' \to (a \to a)) \to b' & \mbox{by Lemma \ref{general_properties2} (\ref{291014_10})} \\
	& = & \{(b'' \to b') \to ((a \to a) \to b')' \}' & \mbox{by (I)} \\
	& = & \{b' \to ((a \to a) \to b')' \}' & \mbox{by Lemma \ref{general_properties_equiv} (\ref{LeftImplicationwithtilde})} \\
	
	& = &  \{((a \to a) \to b')'' \to b \}' & \mbox{by Lemma \ref{properties_of_I20_MC} (\ref{310516_02})} \\
	& = &   \{((a \to a) \to b') \to b \}'  & \mbox{by (\ref{eq_I20})} \\		

	& = & \{((a \to a) \to 0') \to b \}' & \mbox{by Lemma \ref{general_properties2} (\ref{281014_05})} \\
	& = & \{(0 \to (a \to a)') \to b \}' & \mbox{by Lemma \ref{general_properties} (\ref{cuasiConmutativeOfImplic})} \\
	& = &  \{(0 \to (a \to a'')') \to b \}'  &  \mbox{by (\ref{eq_I20})} \\
	& = & \{(0 \to (a' \to a')) \to b \}' & \mbox{by Lemma \ref{general_properties2} (\ref{191114_05})} \\
	& = & \{(a' \to a') \to b \}' & \mbox{by hypothesis} \\
	& = & \{(a \to a) \to b \}' & \mbox{by Lemma \ref{properties_of_I20_MC} (\ref{310516_02})}. 
	\end{array}
	\]
	Therefore,  
	\begin{equation} \rm\label{040716_02}
	\mathbf A \models (x \to x) \to y' \approx ((x \to x) \to y)'.
	\end{equation}
	Hence,
$$
\begin{array}{llll}
	((a \to a) \to b) \to c & =  c' \to [(a \to a) \to b]' & \mbox{by Lemma \ref{properties_of_I20_MC} (\ref{310516_02})} \\
	& =  c' \to [(a \to a) \to b'] & \mbox{by (\ref{040716_02})} \\
	& =  (a \to a) \to (c' \to b') & \mbox{by Lemma \ref{properties_of_I20_MC} (\ref{310516_01})} \\
	& =  (a \to a) \to (b \to c) & \mbox{by Lemma \ref{properties_of_I20_MC} (\ref{310516_02})}. 
\end{array}
$$
	This completes the proof.
\end{proof}

\begin{Lemma} \label{including_relations_of_BA}
	$\mathcal{BA} \subseteq \mathcal A_{12}$.
\end{Lemma}

\begin{proof}
	It is routine to check that $\mathbf{2_b} \in \mathcal A_{12}$. 
	The proof is complete since we know $\mathcal{V}(\mathbf{2_b}) = \mathcal{BA}$.
\end{proof}

\begin{Lemma} \label{including_relations_of_SL}	 
	$\mathcal{SL} \subseteq \mathcal X_{ij}$, for  $\mathcal X \in \{\mathcal A,\mathcal B,\mathcal C,\mathcal D,\mathcal E,\mathcal F\}$ and for all $1 \leq i < j \leq 5$.	
\end{Lemma}

\begin{proof}
	By a routine computation, it is easy to check that $\mathbf{2_s} \in  \mathcal X_{ij}$ for all $1 \leq i < j \leq 5$.
	Then the proof is complete, in view of 
	$\mathcal{V}(\mathbf{2_s}) = \mathcal{SL}$.	
\end{proof}

Let $\mathbf{M}$ denote the set that contains the subvarieties $\mathcal{A}_{13}$, $\mathcal{A}_{14}$, $\mathcal{A}_{15}$, $\mathcal{A}_{24}$, $\mathcal{A}_{34}$, 
$\mathcal{A}_{45}$, $\mathcal{B}_{12}$, $\mathcal{B}_{14}$, $\mathcal{B}_{15}$, $\mathcal{B}_{23}$, $\mathcal{B}_{24}$, $\mathcal{B}_{34}$, $\mathcal{B}_{35}$, $\mathcal{B}_{45}$, $\mathcal{C}_{12}$, $\mathcal{C}_{13}$, $\mathcal{C}_{14}$, $\mathcal{C}_{15}$, $\mathcal{C}_{23}$, $\mathcal{C}_{24}$, $\mathcal{C}_{34}$, $\mathcal{C}_{35}$, $\mathcal{C}_{45}$, $\mathcal{D}_{13}$, $\mathcal{D}_{14}$, $\mathcal{D}_{15}$, $\mathcal{D}_{23}$, $\mathcal{D}_{24}$, 
$\mathcal{D}_{34}$, $\mathcal{D}_{45}$, $\mathcal{E}_{12}$, $\mathcal{E}_{13}$, $\mathcal{E}_{14}$, $\mathcal{E}_{15}$, $\mathcal{E}_{23}$, $\mathcal{E}_{24}$, $\mathcal{E}_{34}$, $\mathcal{E}_{35}$, $\mathcal{E}_{45}$, $\mathcal{F}_{12}$, $\mathcal{F}_{14}$, $\mathcal{F}_{15}$, $\mathcal{F}_{23}$, $\mathcal{F}_{24}$, $\mathcal{F}_{34}$, $\mathcal{F}_{35}$, $\mathcal{F}_{45}$. Observe that $\mathbf{M}$ has $47$ elements.

\begin{Theorem} \label{theo_47_SL}
	If  $\mathcal X \in \mathbf{M}$ then 	$\mathcal X = \mathcal{SL} $. 
\end{Theorem}

\begin{proof}
	In the proof below the following list of statements will be useful.
	
	\medskip
	
	\noindent \begin{minipage}{0.1 \textwidth}
		($\ast$)
	\end{minipage}
	\begin{minipage}{0.9 \textwidth}
		The identity (\ref{eq_I20}), 
		Lemma \ref{general_properties_equiv} (\ref{TXX}),
		Lemma \ref{general_properties_equiv} (\ref{LeftImplicationwithtilde}), 
		Lemma \ref{general_properties2} (\ref{281014_05}),
		Lemma \ref{general_properties3} (\ref{291014_02}),
		Lemma \ref{general_properties3} (\ref{031114_04}) and 
		Lemma \ref{properties_of_I20_MC} (\ref{310516_01}).
	\end{minipage} 
	
	\medskip
	Let $\mathcal X \in M$.  In view of Lemma \ref{including_relations_of_SL}, it suffices to prove that $\mathcal X \subseteq \mathcal{SL}$.  In fact, by Lemma \ref{lemma_SL_I10_C}, it suffices to prove that $\mathcal X \models x' \approx x$.  Let $\mathbf A \in \mathcal X$ and let 
		$a \in \mathbf A$.
	
	To facilitate a uniform presentation (and to make the proof shorter), we introduce the following notation, where $x_0, y_0, z_0 \in \mathbf A$:\\
	
	The notation	
	$$ \mathcal X/ x_0, y_0, z_0$$ denotes the following 
	statement:

	\begin{center}	 
		``Given an algebra $\mathbf A \in \mathcal X$, with $\mathcal X$ being defined by the identity $(X)$, consider $x_0, y_0, z_0 \in \mathbf A$. By substituting $x:= x_0, y:= y_0$ and $z:=z_0$ in $(X)$ and simplifying using the list ($\ast$), we obtain a proof of $\mathbf A \models x \to x \approx x$.''\\ 
	\end{center}
	
For example, suppose we want to prove that the statement ``$\mathcal A_{24}/ a, a', a$'' is true.   We start with $\mathbf A \in \mathcal A_{24}$ and observe that
$$
\begin{array}{lcll}
a & = & a' \to a & \mbox{by Lemma \ref{general_properties_equiv} (\ref{LeftImplicationwithtilde})} \\
& = & (a'' \to a') \to a & \mbox{by Lemma \ref{general_properties_equiv} (\ref{LeftImplicationwithtilde})} \\
& = & (a \to a') \to a & \mbox{by (\ref{eq_I20})} \\
& = & (a \to (a \to a')) \to a & \mbox{by Lemma \ref{general_properties3} (\ref{031114_04})} \\
& = & a \to ((a \to a') \to a) & \mbox{by ($A_{2 4}$)} \\
& = & a \to ((a'' \to a') \to a) & \mbox{by (\ref{eq_I20})} \\
& = & a \to (a' \to a) & \mbox{by Lemma \ref{general_properties_equiv} (\ref{LeftImplicationwithtilde})} \\
& = & a \to a & \mbox{by Lemma \ref{general_properties_equiv} (\ref{LeftImplicationwithtilde}).} 
\end{array}
$$
Consequently, $$\mathbf A \models x \approx x \to x.$$
By Lemma \ref{310516_09}, $$\mathbf A \models x \approx x'.$$

		We divide the 47 varieties under consideration into several groups (again, with a view to making this proof shorter) 
	as follows:

	Firstly, we consider the varieties associated with the following statements:
	
	\begin{multicols}{3}
		
		\begin{enumerate}
			
			\item $\mathcal A_{24}/ a, a', a$,  
			
			\item  $\mathcal A_{45}/ a, 0, 0$,
			
			\item $\mathcal B_{12}/ a',a,0$,
			
			\item $\mathcal B_{14}/ a,a,0$,
			
			\item $\mathcal B_{23}/ a',0,0$,
			
			\item $\mathcal B_{24}/ a, 0', 0$, 
			
			\item $\mathcal B_{34}/ 0', a, a$,  
			
			\item $\mathcal B_{35}/ 0', a, a$,  
			
			\item $\mathcal B_{45}/ a, 0', 0$,

			\item $\mathcal C_{15}/ 0', a, 0$,

			\item $\mathcal C_{23}/ 0', a, 0$,

			\item $\mathcal C_{24}/ a, 0, a$,

			\item $\mathcal C_{35}/ 0', a, 0$,  
			
			\item $\mathcal C_{45}/ a, 0, a$, 
			
			\item $\mathcal D_{13}/ a, a', a'$, 
			
			\item $\mathcal D_{14}/ a, a, a'$, 
			
			\item $\mathcal D_{23}/ a, 0', 0'$, 
			
			\item $\mathcal D_{24}/ a, 0', 0'$, 
			
			\item $\mathcal D_{45}/ a, 0, 0$, 
			
			\item $\mathcal E_{15}/ 0', a, 0'$, 
			
			\item $\mathcal E_{23}/ 0', a, 0'$, 
			
			\item $\mathcal E_{34}/ 0', a, 0'$, 
			
			\item $\mathcal E_{35}/ 0', a, 0'$, 
			
			\item $\mathcal F_{24}/ a, a, 0$. 
		\end{enumerate}
		
	\end{multicols}

	It is routine to verify that each of the above statements is true, from which it follows that, in each case,
	$\mathcal X \models x \to x \approx x$.  Then, applying Lemma \ref{310516_09}, we get that $\mathcal X \models x' \approx x$.  \\

The notation 
		 
	$$\mathcal X / x_0, y_0, z_0 / x_1, y_1, z_1\ /p \approx q$$
	
	is an abbreviation for the following statement: \\

	``Given an algebra $\mathbf A \in \mathcal X$, with $\mathcal X$ being defined by the identity $(X)$, consider $x_0, y_0, z_0 \in \mathbf A$.   By substituting $x:= x_0, y:= y_0, z:= z_0$, 
	and simplifying (X) using (the appropriate lemmas from) the list ($\ast$), we obtain a proof of    $\mathbf A \models 0' \approx 0$.  Then, 
	 by substituting $x:= x_1, y:= y_1, z:= z_1$ in the identity (X) and simplifying it, using $0'=0$ and the list ($\ast$), we obtain a proof of  $\mathbf A \models p \approx q$.''\\

	Secondly, consider the varieties associated with the following statements:
	\begin{multicols}{2}
		
		\begin{enumerate}
			\item $\mathcal D_{15} / 0, 0, 0' / 0, a, a\ / x \to x \approx x$,

			\item $\mathcal D_{34} / 0, 0, 0 / 0', a, a\ / x \to x \approx x $,

			\item $\mathcal E_{13} / 0, 0, 0' / a', 0, a' / x \to x \approx x$,

			\item $\mathcal E_{14} / 0, 0, 0 / a, a', 0\ / x \to x \approx x$,

			\item $\mathcal E_{24} / 0, 0, 0 / a, 0, a\ / x \to x \approx x$,

			\item $\mathcal E_{45} / 0, 0, 0 / a, 0, a\ / x \to x \approx x$.
			
		\end{enumerate}
	\end{multicols}
	
	It is straightforward to verify that each of the above statements is true.  Hence, it follows that  in each case $\mathcal{X} \models  x \to x \approx x$.  Then, applying Lemma \ref{310516_09}, we get that $\mathcal X \models x' \approx x$.

	Thirdly, consider the varieties associated with the following statements:
	
	\begin{multicols}{2}
		\begin{enumerate}
			\item $\mathcal A_{14} / 0,0,0 / a, 0', 0\  / x'  \approx x, $

			\item $\mathcal F_{14} / 0,0,0 / a, 0, 0\  / x'  \approx x, $

			\item $\mathcal F_{23} / 0',0,0 / a, 0, 0\  / x'  \approx x, $

			\item $\mathcal F_{34} / 0',0,0 / a, a, 0\  / x'  \approx x, $

			\item $\mathcal F_{35} / 0',0,0 / 0', a, 0\  / x'  \approx x, $

			\item $\mathcal F_{45} / 0,0,0 / a, 0, 0\  / x'  \approx x. $

		\end{enumerate}
	\end{multicols}
	
	It is easy to verify that the above statements are true.
	Hence, it follows in each of the above cases that $\mathcal X \models x' \approx x.$

	Lastly, consider the varieties associated with the following statements:
	\begin{enumerate}
		\item $\mathcal A_{15} / 0, 0', 0 / a, 0, 0\  / x'  \approx x \to x,$

		\item $\mathcal A_{34} / 0, 0, 0 / a, 0, 0\  / x'  \approx x \to x,$

		\item $\mathcal F_{15} / 0', 0, 0 / a, a', 0\  / x'  \approx x \to x.$

	\end{enumerate}
	
	It is clear that each of the above statements is true.  Hence in each case $\mathcal{X} \models  
	x'  \approx x \to x $.  Then, applying Lemma \ref{010616_01}, we get that $\mathcal X \models x' \approx x$.

	Thus, the varieties still left to consider are $\mathcal A_{13}$, $\mathcal B_{15}$, $\mathcal C_{12}$, $\mathcal C_{13}$, $\mathcal C_{14}$, $\mathcal C_{34}$, $\mathcal E_{12}$ and $\mathcal F_{12}$.
	
	From Lemma \ref{properties_of_I20_MC} (\ref{310516_01}) we can easily verify that 
	$\mathcal C_{13} = \mathcal B_{12}$,
	$\mathcal C_{14} = \mathcal B_{14}$,
	$\mathcal C_{34} = \mathcal B_{24}$,
	$\mathcal E_{12} = \mathcal D_{13}$,
	and
	$\mathcal F_{12} = \mathcal E_{13}$,
	from which we have, in view of earlier conclusions that $\mathcal C_{13} = \mathcal C_{14} =
	\mathcal C_{34} =
	\mathcal E_{12} = 
	\mathcal E_{14} = 
	\mathcal F_{12}  = \mathcal{SL}.$
	
	Now, we are left with $\mathcal A_{13}$, $\mathcal B_{15}$, and $\mathcal C_{12}$ to consider.	
	But, it can be easily seen, using Lemma \ref{properties_of_I20_MC}, that  $\mathcal C_{12} =\mathcal A_{13}$.  Thus, it only remains to verify that 
	$\mathcal A_{13}  = \mathcal{SL}$ and $\mathcal B_{15} =  \mathcal{SL}$.

	First, we show that $\mathcal A_{13}  = \mathcal{SL}$.  
	Let $\mathbf A \in \mathcal A_{13}$ and let $a \in A$.  Then
	
	$$
	\begin{array}{lcll}
	a & = & 0' \to a & \mbox{by Lemma \ref{general_properties_equiv} (\ref{TXX})} \\
	& = & (0 \to 0) \to a & \mbox{} \\
	& = & (0 \to 0) \to (a' \to a) & \mbox{by Lemma \ref{general_properties_equiv} (\ref{LeftImplicationwithtilde})} \\
	& = & 0 \to (0 \to (a' \to a)) & \mbox{by ($A_{1 3}$)} \\ 
	& = & 0 \to (0 \to a) & \mbox{by Lemma \ref{general_properties_equiv} (\ref{LeftImplicationwithtilde})} \\
	& = & 0 \to a & \mbox{by Lemma \ref{general_properties3} (\ref{031114_04})}. 
	\end{array}
	$$	
	Hence
	\begin{equation} \label{300516_01}
	\mathbf A \models x \approx 0 \to x.
	\end{equation}
	Therefore
	$$
	\begin{array}{lcll}
	a & = & a' \to a & \mbox{by Lemma \ref{general_properties_equiv} (\ref{LeftImplicationwithtilde})} \\
	& = & a' \to (a' \to a) & \mbox{by Lemma \ref{general_properties3} (\ref{031114_04})} \\
	& = & a' \to (a' \to a'') & \mbox{by (\ref{eq_I20})} \\
	& = & a' \to (a' \to (a' \to 0)) & \mbox{} \\
	& = & (a' \to a') \to (a' \to 0) & \mbox{by ($A_{1 3}$)} \\ 
	& = & (a' \to a') \to a'' & \mbox{} \\
	& = & (a' \to 0') \to a'' & \mbox{by Lemma \ref{general_properties3} (\ref{281014_05})} \\
	& = & (0 \to a'') \to a'' & \mbox{by Lemma \ref{general_properties} (\ref{cuasiConmutativeOfImplic})} \\
	& = & (0 \to a) \to a & \mbox{by (\ref{eq_I20})} \\
	& = & a \to a  & \mbox{by (\ref{300516_01})}. 
	\end{array}
	$$
	
	Using Lemma \ref{310516_09} we have that $\mathbf A \in \mathcal{SL}$, hence  $\mathcal A_{13} \subseteq \mathcal{SL}$, implying $\mathcal A_{13}  = \mathcal{SL}$.

	Next, we prove that $\mathcal B_{15} =  \mathcal{SL}$.
	Let $\mathbf A \in \mathcal B_{15}$ and let $a \in A$.  Then
	
	$$
	\begin{array}{lcll}
	a & = & a'' & \mbox{by (\ref{eq_I20})} \\
	& = & (a'' \to a')' & \mbox{by Lemma \ref{general_properties_equiv} (\ref{LeftImplicationwithtilde})} \\
	& = & ((a' \to 0) \to a')' & \mbox{} \\
	& = & ((a' \to 0) \to a') \to 0 & \mbox{} \\
	& = & a' \to (0 \to (a' \to 0)) & \mbox{by ($B_{1 5}$)} \\ 
	& = & a' \to (0 \to a'') & \mbox{} \\
	& = & a' \to (0 \to a) & \mbox{by (\ref{eq_I20})} \\
	& = & 0 \to (a' \to a) & \mbox{by Lemma \ref{properties_of_I20_MC} (\ref{310516_01})} \\
	& = & 0 \to a & \mbox{by Lemma \ref{general_properties_equiv} (\ref{LeftImplicationwithtilde})}. 
	\end{array}
	$$
	Hence
	$$
	\begin{array}{lcll}
	a & = & a'' & \mbox{by (\ref{eq_I20})} \\
	& = &   (a'' \to a')' & \mbox{by Lemma \ref{general_properties_equiv} (\ref{LeftImplicationwithtilde})} \\
	& = &   (a \to a')' &  \mbox{by (\ref{eq_I20})} \\
	& = & (a \to (a \to 0))' & \mbox{} \\
	& = & (a \to (a \to (a \to 0)))' & \mbox{by Lemma \ref{general_properties3} (\ref{031114_04})} \\
	& = & ((a \to a) \to a)'' & \mbox{by ($B_{1 5}$)} \\ 
	& = & (a \to a) \to a & \mbox{by (\ref{eq_I20})} \\
	& = & (0 \to a) \to a & \mbox{by Lemma \ref{general_properties2} (\ref{291014_10})} \\
	& = & a \to a & \mbox{by previous calculus.} 
	\end{array}
	$$
	Thus, $$\mathbf A \models x \approx x \to x.$$
	By Lemma \ref{310516_09} we have that $\mathbf A \models x' \approx x$, implying that 
	 $\mathcal B_{15} \subseteq \mathcal{SL}$.  So, $\mathcal B_{15} =  \mathcal{SL}$, 
	completing the proof.
\end{proof}

The following corollary is immediate from the above theorem.

\begin{Corollary}
	$\mathbf M = \{\mathcal{SL}\}$.
\end{Corollary}

Thus, there are 47 Bol-Moufang subvarieties each of which is equal to $\mathcal{SL}$; and hence they are equal to each other.

In order to be able to compare the remaining (possibly distinct) subvarieties with each other, the following lemmas will be useful.

\begin{Lemma} \label{properties_A23}
	Let $\mathbf A \in \mathcal A_{23}$.  Then $\mathbf A $ satisfies the identities:
	\begin{enumerate}[\rm(a)]
		\item $0 \to x \approx x$, \label{310516_03}
		\item $(x \to x) \to y \approx x \to ((x \to 0') \to y)$, \label{310516_05}
		\item $(x \to x) \to y \approx x \to (x \to y')'$. \label{310516_08}
	\end{enumerate}
\end{Lemma}

\begin{proof}
	\begin{enumerate}[(a)]
		\item Let $a \in A$. Then
		$$
		\begin{array}{lcll}
		a	& = & 0' \to a & \mbox{by Lemma \ref{general_properties_equiv} (\ref{TXX})} \\
		& = & (0 \to 0) \to a & \mbox{} \\
		& = & (0 \to 0) \to a'' & \mbox{by (\ref{eq_I20})} \\
		& = & (0 \to 0) \to (a' \to 0) & \mbox{} \\
		& = & 0 \to (0 \to a')' & \mbox{by ($A_{2 3}$)} \\ 
		& = & 0 \to a'' & \mbox{by Lemma \ref{general_properties3} (\ref{031114_07})} \\
		& = & 0 \to a & \mbox{by (\ref{eq_I20})}. 
		\end{array}
		$$
		
		\item Let $a,b \in A$. Then, using ($A_{2 3}$) and Lemma \ref{general_properties_equiv} (\ref{TXX}), we have that $a \to ((a \to 0') \to b) = (a \to a) \to (0' \to b) = (a \to a) \to b$.
		
		\item Let $a,b \in A$.  By ($A_{2 3}$) we have that $(a \to a) \to b = (a \to a) \to b'' = (a \to a) \to (b' \to 0) = a \to (a \to b')'$.
	\end{enumerate}
\end{proof}

\begin{Lemma}  \label{properties_A25} \label{properties_C25} \label{properties_D25}  \label{properties_E25}
We have	
	
	\begin{enumerate}[\rm(1)]
		\item	Let $\mathbf A \in \mathcal A_{25} \cup \mathcal C_{25} \cup \mathcal D_{25} \cup \mathcal E_{25}$.  Then $\mathbf A $ satisfies the identities
		\begin{enumerate}[\rm(a)]
			\item $0 \to x \approx x$,  \label{060616_01}  \label{070616_03}  \label{070616_05}  \label{070616_07} 
			\item $(x \to y)' \approx x' \to y'$. \label{060616_02}  \label{070616_04}  \label{070616_06}  \label{070616_08}
		\end{enumerate}	
		\item Let $\mathbf A \in \mathcal A_{25}$.  Then $\mathbf A \models x \to (y \to (y' \to z)) \approx y \to ((y \to x) \to z)$. \label{060616_03}
	\end{enumerate}
\end{Lemma}

\begin{proof}
	\begin{enumerate}[\rm(1)]
		\item
		\begin{enumerate}[(a)]
			\item 
			Let $\mathbf A \in \mathcal A_{25}$ and $a \in A$. 
			Then
			$$
			\begin{array}{lcll}
			a	& = & a'' & \mbox{by (\ref{eq_I20})} \\
			& = & a' \to 0 & \mbox{} \\
			& = & (0' \to a') \to 0 & \mbox{by Lemma \ref{general_properties_equiv} (\ref{TXX})} \\
			& = & ((0 \to 0) \to a') \to 0 & \mbox{} \\
			& = & 0 \to ((0 \to a') \to 0) & \mbox{by ($A_{2 5}$)} \\ 
			& = & 0 \to a'' & \mbox{by Lemma \ref{general_properties3} (\ref{031114_07})} \\
			& = & 0 \to a & \mbox{by (\ref{eq_I20})}. 
			\end{array}
			$$	
			
			Let $\mathbf A \in \mathcal C_{25}$ and $a \in A$. 
			Then
			$$
			\begin{array}{lcll}
			0 \to a	& = & 0 \to (0' \to a) & \mbox{by Lemma \ref{general_properties_equiv} (\ref{TXX})} \\
			& = & 0 \to ((0' \to 0') \to a) & \mbox{by Lemma \ref{general_properties_equiv} (\ref{TXX})} \\
			& = & ((0 \to 0') \to 0') \to a & \mbox{by ($C_{2 5}$)} \\ 
			& = &  ((0'' \to 0') \to 0') \to a &  \mbox{by $0 \approx 0''$} \\ 
			& = & (0' \to 0') \to a & \mbox{by Lemma \ref{general_properties_equiv} (\ref{LeftImplicationwithtilde})} \\
			& = & 0' \to a & \mbox{by Lemma \ref{general_properties_equiv} (\ref{TXX})} \\
			& = & a & \mbox{by Lemma \ref{general_properties_equiv} (\ref{TXX}).} 
			\end{array}
			$$

			Let $\mathbf A \in \mathcal D_{25}$ and $a \in A$. 
			Then
			$$
			\begin{array}{lcll}
			0 \to a & = & 0 \to a'' & \mbox{by (\ref{eq_I20})} \\
			& = & 0 \to (0 \to a')' & \mbox{by Lemma \ref{general_properties3} (\ref{031114_07})} \\
			& = & 0 \to ((0 \to a') \to 0) & \mbox{} \\
			& = & ((0 \to 0) \to a') \to 0 & \mbox{by ($D_{2 5}$)} \\ 
			& = & (0' \to a') \to 0 & \mbox{} \\
			& = & a' \to 0 & \mbox{by Lemma \ref{general_properties_equiv} (\ref{TXX})} \\
			& = & a & \mbox{by (\ref{eq_I20})}. 
			\end{array}
			$$

 If $\mathbf A \in \mathcal E_{25}$, then the proof is analogous to that of the preceding case.  Just replace ($D_{2 5}$) with ($E_{2 5}$).

			\item By (\ref{060616_01}), $\mathbf A \models x \approx 0 \to x$. Then by Lemma \ref{lemma_070616_01}, we have
			$$\mathbf A \models (x \to y)' \approx x' \to y'.$$
		\end{enumerate}
		\item Let $a,b,c \in A$.
		Then
		$$
		\begin{array}{lcll}
		a \to (b \to (b' \to c))	& = & a \to ((b \to b)' \to c) & \mbox{by ($A_{2 5}$)} \\ 
		& = & a \to (c' \to (b \to b)) & \mbox{by Lemma \ref{properties_of_I20_MC} (\ref{310516_02})} \\
		& = & a \to (b \to (c' \to b)) & \mbox{by Lemma \ref{properties_of_I20_MC} (\ref{310516_01})} \\
		& = & b \to (a \to (c' \to b)) & \mbox{by Lemma \ref{properties_of_I20_MC} (\ref{310516_01})} \\
		& = & b \to (c' \to (a \to b)) & \mbox{by Lemma \ref{properties_of_I20_MC} (\ref{310516_01})} \\
		& = & b \to ((a \to b)' \to c) & \mbox{by Lemma \ref{properties_of_I20_MC} (\ref{310516_02})} \\
		& = & b \to ((a' \to b') \to c) & \mbox{by (\ref{060616_02})} \\
		& = & b \to ((b \to a) \to c) & \mbox{ by Lemma \ref{properties_of_I20_MC} (\ref{021017_01})}. 
		\end{array}
		$$
	\end{enumerate}
	This completes the proof.
\end{proof}

\begin{Lemma} \label{properties_A12}
	Let $\mathbf A \in \mathcal A_{12} \cup \mathcal D_{12} \cup \mathcal D_{35}$ then $\mathbf A $ satisfies the identities
	\begin{enumerate}[\rm(a)]
		\item $0 \to x' \approx x \to x \approx 0 \to x$,  \label{140616_01} \label{140616_02} \label{140616_03}
		\item $0 \to (x \to y) \approx 0 \to (y \to x)$,  \label{140616_04}
		\item $0 \to (x \to (y \to z)) \approx 0 \to ((x \to y) \to z)$. \label{230616_03}
	\end{enumerate}
\end{Lemma}

\begin{proof} Let $a, b, c \in A$.
	\begin{enumerate}[(a)]
		\item 
		Suppose $\mathbf A \in \mathcal A_{12}$.  Then 
		
		$$
		\begin{array}{lcll}
		0 \to a'	& = & a \to 0' & \mbox{by Lemma \ref{general_properties} (\ref{cuasiConmutativeOfImplic})} \\
		& = & a \to (a \to 0') & \mbox{by Lemma \ref{general_properties2} (\ref{031114_04})} \\
		& = & a \to (a \to (0 \to 0)) & \mbox{} \\
		& = & a \to ((a \to 0) \to 0) & \mbox{by ($A_{1 2}$)} \\ 
		& = & a \to a'' & \mbox{} \\
		& = & a \to a & \mbox{by (\ref{eq_I20})}. 
		\end{array}
		$$
		Thus, we have 
		\begin{equation} \label{230616_01}
		\mathbf A \models 0 \to x' \approx x \to x.	
		\end{equation}
		Hence,

		$$
		\begin{array}{lcll}
		0 \to a & = & 0 \to a'' & \mbox{by (\ref{eq_I20})} \\
		& = & a' \to a' & \mbox{by (\ref{230616_01})} \\
		& = & a \to a & \mbox{ by Lemma \ref{properties_of_I20_MC} (\ref{021017_01})}. 
		\end{array}
		$$
		Consequently, $$\mathbf A \models 0 \to x \approx x \to x.$$

		Next, assume that $\mathbf A \in \mathcal D_{12}$. Then
		
		$$
		\begin{array}{lcll}
		a \to a & = & a \to (0' \to a) & \mbox{by Lemma \ref{general_properties_equiv} (\ref{TXX})} \\
		& = & a \to ((0 \to 0) \to a) & \mbox{} \\
		& = & a \to (0 \to (0 \to a)) & \mbox{by ($D_{1 2}$)} \\ 
		& = & 0 \to (0 \to (a \to a)) & \mbox{by Lemma \ref{properties_of_I20_MC} (\ref{310516_01}) twice} \\
		& = & 0 \to (a \to a) & \mbox{by Lemma \ref{general_properties2} (\ref{031114_04})} \\
		& = & 0 \to (a'' \to a'') & \mbox{by (\ref{eq_I20})} \\
		& = & 0 \to (a' \to a')' & \mbox{by Lemma \ref{general_properties2} (\ref{191114_05})} \\
		& = & 0 \to ((a' \to a') \to 0) & \mbox{} \\
		& = & 0 \to (a' \to (a' \to 0)) & \mbox{by ($D_{1 2}$)} \\ 
		& = & 0 \to (a' \to a) & \mbox{by (\ref{eq_I20})} \\
		& = & 0 \to a & \mbox{by Lemma \ref{general_properties_equiv} (\ref{LeftImplicationwithtilde})}. 
		\end{array}
		$$
		Hence 
		\begin{equation} \label{230616_02}
		\mathbf A \models 0 \to x \approx x \to x.	
		\end{equation}
		Now,
		$$
		\begin{array}{lcll}
		0 \to a' & = & a' \to a' & \mbox{by (\ref{230616_02})} \\
		& = & a \to a & \mbox{ by Lemma \ref{properties_of_I20_MC} (\ref{021017_01})}. 
		\end{array}
		$$
		Consequently, $$\mathbf A \models 0 \to x' \approx x \to x.$$

		Finally, assume that $\mathbf A \in \mathcal D_{35}$.  Then 
		$$
		\begin{array}{lcll}
		0 \to a & = & a' \to (0 \to a) & \mbox{by Lemma \ref{general_properties2} (\ref{281114_01})} \\
		& = & (a \to 0) \to (0 \to a) & \mbox{} \\
		& = & ((a \to 0) \to 0) \to a & \mbox{by ($D_{3 5}$)} \\ 
		& = & a'' \to a & \mbox{} \\
		& = & a \to a & \mbox{by (\ref{eq_I20})}. 
		\end{array}
		$$
		Hence 
		\begin{equation} \label{230616_05}
		\mathbf A \models 0 \to x \approx x \to x.	
		\end{equation}
		$$
		\begin{array}{lcll}
		0 \to a' & = & a' \to a' & \mbox{by (\ref{230616_05})} \\
		& = & a \to a & \mbox{ by Lemma \ref{properties_of_I20_MC} (\ref{021017_01})}. 
		\end{array}
		$$
		Consequently, $$\mathbf A \models 0 \to x' \approx x \to x.$$

		\item Observe that
		
		$$
		\begin{array}{lcll}
		0 \to (a \to b) & = & 0 \to (b' \to a') & \mbox{ by Lemma \ref{properties_of_I20_MC} (\ref{021017_01})} \\
		& = & (0 \to b') \to (0 \to a') & \mbox{by Lemma \ref{general_properties2} (\ref{071114_04}) and (\ref{311014_06})} \\
		& = & (0 \to b) \to (0 \to a) & \mbox{by (\ref{140616_03})} \\
		& = & 0 \to (b \to a) & \mbox{by Lemma \ref{general_properties2} (\ref{071114_04}) and (\ref{311014_06})}. 
		\end{array}
		$$

		\item
		
		$$
		\begin{array}{lcll}
		0 \to (b \to (c \to a)) & = & (0 \to b) \to (0 \to (c \to a)) &  \\
		&  & \mbox{by Lemma \ref{general_properties2} (\ref{071114_04}) and (\ref{311014_06})} &  \\
		& = & (0 \to b) \to (0 \to (a \to c)) & \mbox{by (\ref{140616_04})} \\
		& = & 0 \to (b \to (a \to c)) &  \\
		&  & \mbox{by Lemma \ref{general_properties2} (\ref{071114_04}) and (\ref{311014_06})} &  \\
		& = & 0 \to (a \to (b \to c)) & \mbox{by Lemma \ref{properties_of_I20_MC} (\ref{310516_01})} \\
		& = & 0 \to ((b \to c) \to a) & \mbox{by  (\ref{140616_04})}. 
		\end{array}
		$$	
	\end{enumerate}
	This completes the proof.
\end{proof}

\vspace{1cm}
\section{Distinct Varieties of Symmetric $\mathcal{I}$-zroupoids of Bol-Moufang type}

Recall that 47 of the 60 subvarieties of $\mathcal{S}$ were shown to be equal to the variety $\mathcal{SL}$.  In this section we will investigate the relationships among the remaining 13 varieties: $\mathcal A_{12}, \mathcal A_{23},  \mathcal A_{25},  \mathcal A_{35},  \mathcal B_{13}, \mathcal B_{25}, \mathcal C_{25},  
\mathcal D_{12},$
$ \mathcal D_{25},  \mathcal D_{35}, \mathcal E_{25},
\mathcal F_{13}$ and
$\mathcal F_{25} $.    
We still need to determine which of these are distinct from each other.  The following theorem throws more light on this issue.  

\begin{Theorem} \label{theorem_040716_05}
	We have
	\begin{enumerate}[\rm(a)]
		\item $\mathcal A_{23} = \mathcal A_{25} = \mathcal C_{25} = \mathcal D_{25} = \mathcal E_{25}$,
		\item $\mathcal A_{12} = \mathcal B_{13}  =  \mathcal D_{12} = \mathcal D_{35} = \mathcal F_{13}$,
		\item $\mathcal A_{35} = \mathcal F_{25}$, \label{040716_06}
		 \item $\mathcal B_{25} \ = \  \mathcal{S}$.
	\end{enumerate}
\end{Theorem}

\begin{proof}
	\begin{enumerate}[(a)]
		\item 
		
		Let $\mathbf A \in \mathcal{A}_{23}$ and $a,b,c \in A$. We have that
		$$
		\begin{array}{lcll}
		c' \to (a \to (b \to a)) & = & a \to (c' \to (b \to a)) & \mbox{by Lemma \ref{properties_of_I20_MC} (\ref{310516_01})} \\
		& = & a \to (b \to (c' \to a))  & \mbox{by Lemma \ref{properties_of_I20_MC} (\ref{310516_01})} \\
		& = & a \to (b \to (a' \to c)) & \mbox{by Lemma \ref{properties_of_I20_MC} (\ref{310516_02})} \\
		& = & a \to (a' \to (b \to c)). & \mbox{by Lemma \ref{properties_of_I20_MC} (\ref{310516_01})}. 
		\end{array}
		$$
		Hence $\mathbf A$ satisfies the identity
		\begin{equation} \label{310516_04}
		z' \to (x \to (y \to x)) \approx x \to (x' \to (y \to z)).
		\end{equation}
		
		Observe	
		$$
		\begin{array}{lcll}
		(a \to (a \to b')')' & = & ((a \to b') \to a')' & \mbox{by Lemma \ref{properties_of_I20_MC} (\ref{021017_02})} \\
		& = & ((a'' \to a) \to (b' \to a')')'' & \mbox{by (I)} \\
		& = & (a \to a) \to (b' \to a')'  & \mbox{by (\ref{eq_I20})} \\
		& = & (a \to a) \to ((b \to 0) \to a')' & \mbox{} \\
		& = & (a \to a) \to ((a'' \to b) \to (0 \to a')')'' & \mbox{by (I)} \\
		& = & (a \to a) \to ((a \to b) \to (0 \to a')') & \mbox{by (\ref{eq_I20})} \\
		& = & (a \to a) \to ((a \to b) \to a'') & \mbox{by Lemma \ref{properties_A23} (\ref{310516_03})} \\
		& = & (a \to a) \to ((a \to b) \to a) & \mbox{by (\ref{eq_I20})} \\
		& = & (a \to a) \to ((0 \to b) \to a) & \mbox{by Lemma \ref{general_properties2} (\ref{291014_10})} \\
		& = & (a \to a) \to (b \to a) & \mbox{by Lemma \ref{properties_A23} (\ref{310516_03})} \\
		& = & a \to ((a \to b) \to a) & \mbox{by ($A_{2 3}$)} \\ 
		& = &  a \to ((0 \to b) \to a) & \mbox{by Lemma \ref{general_properties2} (\ref{291014_10})} \\
		& = & a \to (b \to a) & \mbox{by Lemma \ref{properties_A23} (\ref{310516_03})}. 
		\end{array}
		$$
		Hence, using (\ref{310516_04}), we have that $\mathbf A$ satisfies
		\begin{equation*}
		z' \to (x \to (x \to y')')' \approx x \to (x' \to (y \to z)).
		\end{equation*}
		Therefore, by Lemma \ref{properties_A23} (\ref{310516_03}),
		\begin{equation} \label{310516_06}
		\mathbf A \models z' \to (x \to (x \to y')')' \approx x \to ((0 \to x') \to (y \to z)).
		\end{equation}
		From
		$$
		\begin{array}{lcll}
		a \to ((0 \to a') \to (b \to c)) & = & a \to ((a \to 0') \to (b \to c)) & \mbox{by Lemma \ref{general_properties} (\ref{cuasiConmutativeOfImplic})} \\
		& = & (a \to a) \to (b \to c) & \mbox{by Lemma \ref{properties_A23} (\ref{310516_05})} \\
		& = & a \to ((a \to b) \to c) & \mbox{by ($A_{2 3}$)}, \\  
		\end{array}
		$$
		and the equation (\ref{310516_06}), we obtain
		\begin{equation} \label{310516_07}
		\mathbf A \models z' \to (x \to (x \to y')')' \approx x \to ((x \to y) \to z).
		\end{equation}
		Finally, 
		$$
		\begin{array}{lcll}
		a \to ((a \to b) \to c)& = & c' \to (a \to (a \to b')')' & \mbox{by (\ref{310516_07})} \\
		& = & c' \to ((a \to a) \to b)' & \mbox{by Lemma \ref{properties_A23} (\ref{310516_08})} \\
		& = & ((a \to a) \to b) \to c & \mbox{by Lemma \ref{properties_of_I20_MC} (\ref{021017_01})}, 
		\end{array}
		$$
		proving that the identity ($A_{2 5}$) holds in $\mathbf A$.
		Therefore, we have that
		\begin{equation} \label{070616_09}
		\mathcal{A}_{23} \subseteq \mathcal{A}_{25}.
		\end{equation}	
		Let $\mathbf A \in \mathcal{A}_{25}$ and $a,b,c \in A$. Then
$$
\begin{array}{llll}
		((a \to b) \to b) \to c & =  (((a \to b) \to b) \to c)'' & \mbox{by (\ref{eq_I20})} \\
		& =  (c' \to ((a \to b) \to b)')'' & \mbox{ by Lemma \ref{properties_of_I20_MC} (\ref{021017_01})} \\
		& =  (c' \to (b' \to (a \to b)')')'' & \mbox{{by Lemma \ref{properties_of_I20_MC} (\ref{021017_01})}} \\
		& =  (c' \to (b \to (a \to b)))'' & \mbox{by Lemma \ref{properties_A25} (\ref{060616_02})} \\
		& =  (c'' \to (b \to (a \to b))')' &  \mbox{by Lemma \ref{properties_A25} (\ref{060616_02})} \\
		& =  (c \to (b \to (a \to b))')'  & \mbox{by (\ref{eq_I20})} \\
		& =  (c \to (b \to (b' \to a'))')' & \mbox{by Lemma \ref{properties_of_I20_MC} (\ref{021017_01})} \\
		& =  (c \to (b' \to (b' \to a')'))' & \mbox{by Lemma \ref{properties_A25} (\ref{060616_02})} \\
		& =  (c \to (b' \to (b \to a)))' &  \mbox{by Lemma \ref{properties_A25} (\ref{060616_02})} \\
		& =  (c \to (b \to (b' \to a)))' & \mbox{by Lemma \ref{properties_of_I20_MC} (\ref{310516_01})} \\
		& =  (b \to ((b \to c) \to a))' & \mbox{by Lemma \ref{properties_A25} (\ref{060616_03})} \\
		& =  (((b \to b) \to c) \to a)' & \mbox{by ($A_{2 5}$)} \\ 
		& =  (a' \to ((b \to b) \to c)')' & \mbox{by Lemma \ref{properties_of_I20_MC} (\ref{021017_01})} \\
		& =  a \to ((b \to b) \to c) & \mbox{by Lemma \ref{properties_A25} (\ref{060616_02})}. 
\end{array}
$$

		Therefore, we have that
		\begin{equation} \label{070616_10}
		\mathcal{A}_{25} \subseteq \mathcal{C}_{25}.
		\end{equation}

		Let $\mathbf A \in \mathcal{C}_{25}$ and $a,b,c \in A$. We have that
		$$
		\begin{array}{lcll}
		((a \to b) \to c) \to a & = & ((a \to b) \to c)'' \to a & \mbox{by (\ref{eq_I20})} \\
		& = & ((a \to b)' \to c')' \to a & \mbox{by Lemma \ref{properties_C25} (\ref{070616_04})} \\
		& = & ((a' \to b') \to c')' \to a & \mbox{by Lemma \ref{properties_C25} (\ref{070616_04})} \\
		& = & a' \to ((a' \to b') \to c') & \mbox{by Lemma \ref{properties_of_I20_MC} (\ref{310516_02})} \\
		& = & a' \to (c \to (a' \to b')') & \mbox{by Lemma \ref{properties_of_I20_MC} (\ref{021017_02})} \\
		& = & c \to (a' \to (a' \to b')') & \mbox{by Lemma \ref{properties_of_I20_MC} (\ref{310516_01})} \\
		& = & c \to (a' \to (a \to b)) & \mbox{by Lemma \ref{properties_C25} (\ref{070616_04})}. \\
		 & = & c \to (a' \to (a'' \to b)) & \mbox{by (\ref{eq_I20}) }\\
		& = & c \to (a' \to (b' \to a')) & \mbox{by Lemma \ref{properties_of_I20_MC} (\ref{310516_02})}\\ 
		& = & c \to (b' \to (a' \to a')) & \mbox{by Lemma \ref{properties_of_I20_MC} (\ref{310516_01})} \\
		& = & c \to ((a' \to a')' \to b) & \mbox{by Lemma \ref{properties_of_I20_MC} (\ref{310516_02})} \\
		& = & c \to ((a \to a) \to b) & \mbox{by Lemma \ref{properties_C25} (\ref{070616_04})} \\
		& = & ((c \to a) \to a) \to b & \mbox{by ($C_{2 5}$)} \\ 
		& = & b' \to ((c \to a) \to a)' & \mbox{by Lemma \ref{properties_of_I20_MC} (\ref{021017_01})} \\
		& = & b' \to ((c \to a)' \to a') & \mbox{by Lemma \ref{properties_C25} (\ref{070616_04})} \\
		& = & b' \to (a \to (c \to a)) & \mbox{by Lemma \ref{properties_of_I20_MC} (\ref{021017_01})} \\
		& = & b' \to (c \to (a \to a)) & \mbox{by Lemma \ref{properties_of_I20_MC} (\ref{310516_01})} \\
		& = & b' \to ((a \to a)' \to c') & \mbox{by Lemma \ref{properties_of_I20_MC} (\ref{021017_01})} \\
		& = & (a \to a)' \to (b' \to c') & \mbox{by Lemma \ref{properties_of_I20_MC} (\ref{310516_01})} \\
		& = & (b' \to c')' \to (a \to a) & \mbox{by Lemma \ref{properties_of_I20_MC} (\ref{310516_02})} \\
		& = & (b \to c) \to (a \to a) & \mbox{by Lemma \ref{properties_C25} (\ref{070616_04})} \\
		& = & a \to ((b \to c) \to a) & \mbox{by Lemma \ref{properties_of_I20_MC} (\ref{310516_01})}. 
		\end{array}
		$$
		Thus,
		\begin{equation} \label{070616_11}
		\mathcal{C}_{25} \subseteq \mathcal{D}_{25}.
		\end{equation}

		Let $\mathbf A \in \mathcal{D}_{25}$ and $a,b,c \in A$. Then
		
		$$
		\begin{array}{lcll}
		((a \to b) \to c) \to b & = & b' \to ((a \to b) \to c)' & \mbox{by Lemma \ref{properties_of_I20_MC} (\ref{021017_01})} \\
		& = & b' \to (c' \to (a \to b)')' & \mbox{by Lemma \ref{properties_of_I20_MC} (\ref{021017_01})} \\
		& = & b' \to (c' \to (a' \to b'))' & \mbox{by Lemma \ref{properties_D25} (\ref{070616_06})} \\
		& = & b' \to (c \to (a' \to b')')  & \mbox{by Lemma \ref{properties_D25} (\ref{070616_06})} \\
		& = & b' \to (c \to (a \to b)) & \mbox{by Lemma \ref{properties_D25} (\ref{070616_06})} \\
		& = &  b' \to (c \to (a'' \to b)) &  \mbox{by (\ref{eq_I20}) }\\
		& = & b' \to (c \to (b' \to a'))  & \mbox{by Lemma \ref{properties_of_I20_MC} (\ref{310516_02})} \\ 
		& = & b' \to (b' \to (c \to a'))  & \mbox{by Lemma \ref{properties_of_I20_MC} (\ref{310516_01})} \\
		& = & b' \to (c \to a') & \mbox{by Lemma \ref{general_properties2} (\ref{031114_04})} \\
		& = & b' \to (a \to c') & \mbox{by Lemma \ref{properties_of_I20_MC} (\ref{021017_02})} \\
		& = & a \to (b' \to c') & \mbox{by Lemma \ref{properties_of_I20_MC} (\ref{310516_01})} \\
		& = & a \to (b \to c)' & \mbox{by Lemma \ref{properties_D25} (\ref{070616_06})} \\
		& = & (b \to c) \to a' & \mbox{by Lemma \ref{properties_of_I20_MC} (\ref{021017_02})} \\
		& = & (b \to (b \to c)) \to a' & \mbox{by Lemma \ref{general_properties2} (\ref{031114_04})} \\
		& = & ((b \to c)' \to b') \to a' & \mbox{by Lemma \ref{properties_of_I20_MC} (\ref{021017_01})} \\
		& = & ((b \to c) \to b)' \to a' & \mbox{by Lemma \ref{properties_D25} (\ref{070616_06})} \\
		& = & a \to ((b \to c) \to b) & \mbox{by Lemma \ref{properties_of_I20_MC} (\ref{021017_01})}. 
		\end{array}
		$$
		Consequently,
		\begin{equation} \label{070616_12}
		\mathcal{D}_{25} \subseteq \mathcal{E}_{25}.
		\end{equation}

		Let $\mathbf A \in \mathcal{E}_{25}$ and $a,b,c \in A$. Then
		
		$$
		\begin{array}{lcll}
		(a \to a) \to (b \to c) & = & b \to ((a \to a) \to c) & \mbox{by Lemma \ref{properties_of_I20_MC} (\ref{310516_01})} \\
		& = & b \to (c' \to (a \to a)') & \mbox{by Lemma \ref{properties_of_I20_MC} (\ref{021017_01})} \\
		& = & b \to (c' \to (a' \to a')) & \mbox{by Lemma \ref{properties_E25} (\ref{070616_08})} \\
		& = & b \to (a' \to (c' \to a')) & \mbox{by Lemma \ref{properties_of_I20_MC} (\ref{310516_01})}. \\
		& = & b \to ((c' \to a')' \to a) & \mbox{by Lemma \ref{properties_of_I20_MC} (\ref{310516_02})} \\
		& = & b \to ((c \to a) \to a) & \mbox{by Lemma \ref{properties_E25} (\ref{070616_08})} \\
		& = & ((b \to c) \to a) \to a & \mbox{by ($E_{2 5}$)} \\ 
		& = & a' \to ((b \to c) \to a)' & \mbox{by Lemma \ref{properties_of_I20_MC} (\ref{021017_01})} \\
		& = & a' \to (a' \to (b \to c)')' & \mbox{by Lemma \ref{properties_of_I20_MC} (\ref{021017_01})} \\
		& = & a' \to (a' \to (b' \to c'))' & \mbox{by Lemma \ref{properties_E25} (\ref{070616_08})} \\
		& = & a' \to (a' \to (c \to b))' & \mbox{by Lemma \ref{properties_of_I20_MC} (\ref{021017_01})} \\
		& = & a' \to (a \to (c \to b)') & \mbox{by Lemma \ref{properties_E25} (\ref{070616_08})} \\
		& = & a \to (a' \to (c \to b)') & \mbox{by Lemma \ref{properties_of_I20_MC} (\ref{310516_01})} \\
		& = & a \to (a \to (c \to b))' & \mbox{by Lemma \ref{properties_E25} (\ref{070616_08})} \\
		& = & a \to (c \to (a \to b))' & \mbox{by Lemma \ref{properties_of_I20_MC} (\ref{310516_01})} \\
		& = & a \to (c \to (a' \to b')')' & \mbox{by Lemma \ref{properties_E25} (\ref{070616_08})} \\
		& = & a \to ((a' \to b') \to c')' & \mbox{by Lemma \ref{properties_of_I20_MC} (\ref{021017_02})} \\
		& = & a \to ((a \to b)' \to c')' & \mbox{by Lemma \ref{properties_E25} (\ref{070616_08})} \\
		& = & ((a \to b)' \to c') \to a' & \mbox{by Lemma \ref{properties_of_I20_MC} (\ref{021017_02})} \\
		& = & ((a \to b) \to c)' \to a' & \mbox{by Lemma \ref{properties_E25} (\ref{070616_08})} \\
		& = & a \to ((a \to b) \to c) & \mbox{by Lemma \ref{properties_of_I20_MC} (\ref{021017_01})}. 
		\end{array}
		$$
		Therefore,
		\begin{equation} \label{070616_13}
		\mathcal{E}_{25} \subseteq \mathcal{A}_{23}.
		\end{equation}
		
		From (\ref{070616_09}), (\ref{070616_10}), (\ref{070616_11}), (\ref{070616_12}) and (\ref{070616_13}), we conclude that
		$$\mathcal A_{23} = \mathcal A_{25} = \mathcal C_{25} = \mathcal D_{25} = \mathcal E_{25}.$$
		
		\item 
		Using Lemma \ref{properties_of_I20_MC} (\ref{310516_01}) we can easily verify that $\mathcal A_{12} = \mathcal B_{13}$ and $\mathcal F_{13} = \mathcal D_{12}$.
		
		Let $\mathbf A \in \mathcal{A}_{12}$ and $a,b,c \in A$. Then
		$$
		\begin{array}{lcll}
		a \to (b \to (c \to a)) & = & b \to (a \to (c \to a)) & \mbox{by Lemma \ref{properties_of_I20_MC} (\ref{310516_01})} \\
		& = & b \to (c \to (a \to a)) & \mbox{by Lemma \ref{properties_of_I20_MC} (\ref{310516_01})} \\
		& = & b \to (c \to (0 \to a)) & \mbox{by Lemma \ref{properties_A12} (\ref{140616_02})} \\
		& = & b \to (0 \to (c \to a))  & \mbox{by Lemma \ref{properties_of_I20_MC} (\ref{310516_01})} \\
		& = & 0 \to (b \to (c \to a)) & \mbox{by Lemma \ref{properties_of_I20_MC} (\ref{310516_01})} \\
		& = & 0 \to ((b \to c) \to a) & \mbox{by Lemma \ref{properties_A12} (\ref{230616_03})} \\
		& = & (b \to c) \to (0 \to a) & \mbox{by Lemma \ref{properties_of_I20_MC} (\ref{310516_01})} \\
		& = & (b \to c) \to (a \to a) & \mbox{by Lemma \ref{properties_A12} (\ref{140616_02})} \\
		& = & a \to ((b \to c) \to a) & \mbox{by Lemma \ref{properties_of_I20_MC} (\ref{310516_01})}. 
		\end{array}
		$$
		Therefore,
		\begin{equation} \label{140616_05}
		\mathcal{A}_{12} \subseteq \mathcal{D}_{12}.
		\end{equation}
		
		Let $\mathbf A \in \mathcal{D}_{12}$ and $a,b,c \in A$. Then		
		
		$$
		\begin{array}{lcll}
		& & ((a \to b) \to c) \to a & \mbox{} \\
		& = & ((a' \to (a \to b)) \to (c \to a)')' & \mbox{by (I)} \\
		& = & ((a \to (a' \to b)) \to (c \to a)')' & \mbox{by Lemma \ref{properties_of_I20_MC} (\ref{310516_01})} \\
		& = & ((a \to (b' \to a)) \to (c \to a)')' & \mbox{by Lemma \ref{properties_of_I20_MC} (\ref{310516_02})} \\
		& = & ((b' \to (a \to a)) \to (c \to a)')' & \mbox{by Lemma \ref{properties_of_I20_MC} (\ref{310516_01})} \\
		& = & ((b' \to (0 \to a)) \to (c \to a)')' & \mbox{by Lemma \ref{properties_A12} (\ref{140616_02})} \\
		& = & ((0 \to (b' \to a)) \to (c \to a)')' & \mbox{by Lemma \ref{properties_of_I20_MC} (\ref{310516_01})} \\
		& = & ((0 \to (a' \to b)) \to (c \to a)')' & \mbox{by Lemma \ref{properties_of_I20_MC} (\ref{310516_02})} \\
		& = & ((0 \to (0 \to (a' \to b))) \to (c \to a)')' & \mbox{by Lemma \ref{general_properties2} (\ref{031114_04})} \\
		& = & (((c \to a)' \to (0 \to (a' \to b))) \to (c \to a)')' & \mbox{by Lemma \ref{general_properties2} (\ref{291014_10})} \\
		& = & ((0 \to ((c \to a)' \to (a' \to b))) \to (c \to a)')' & \mbox{by Lemma \ref{properties_of_I20_MC} (\ref{310516_01})} 
		\end{array}
		$$
		and
		$$
		\begin{array}{lcll}
		& & ((0 \to ((c \to a)' \to (a' \to b))) \to (c \to a)')' \\
		& = & (((0 \to (c \to a)') \to (0 \to (a' \to b))) \to (c \to a)')' &  \\
		&  & \mbox{by Lemma \ref{general_properties2} (\ref{071114_04}) and (\ref{311014_06})} &  \\
		& = & (((0 \to (c \to a)) \to (0 \to (a' \to b))) \to (c \to a)')' & \mbox{by Lemma \ref{properties_A12} (\ref{140616_02})} \\
		& = & ((0 \to ((c \to a) \to (a' \to b))) \to (c \to a)')' &  \\
		&  & \mbox{by Lemma \ref{general_properties2} (\ref{071114_04}) and (\ref{311014_06})} &  \\
		& = & ((0 \to (((c \to a) \to a') \to b)) \to (c \to a)')' & \mbox{by Lemma \ref{properties_A12} (\ref{230616_03})} \\
		& = & (((0 \to ((c \to a) \to a')) \to (0 \to b)) \to (c \to a)')' &  \\
		&  & \mbox{by Lemma \ref{general_properties2} (\ref{071114_04}) and (\ref{311014_06})} &  \\
		& = & (((0 \to (a' \to (c \to a))) \to (0 \to b)) \to (c \to a)')' & \mbox{by Lemma \ref{properties_A12} (\ref{140616_04})} \\
		& = & (((0 \to (c \to a)) \to (0 \to b)) \to (c \to a)')' & \mbox{by Lemma \ref{general_properties2} (\ref{281114_01})} \\
		& = & (((0 \to (c \to a)') \to (0 \to b)) \to (c \to a)')' & \mbox{by Lemma \ref{properties_A12} (\ref{140616_02})} \\
		& = & ((0 \to ((c \to a)' \to b)) \to (c \to a)')' &  \\
		&  & \mbox{by Lemma \ref{general_properties2} (\ref{071114_04}) and (\ref{311014_06})} &  \\
		& = & (((c \to a)' \to (0 \to b)) \to (c \to a)')' & \mbox{by Lemma \ref{properties_of_I20_MC} (\ref{310516_01})} \\
		& = & ((0 \to (0 \to b)) \to (c \to a)')'  & \mbox{by Lemma \ref{general_properties2} (\ref{291014_10})} \\
		& = & ((0 \to (0 \to b)') \to (c \to a)')' & \mbox{by Lemma \ref{properties_A12} (\ref{140616_02})} \\
		& = & (((0 \to b) \to 0') \to (c \to a)')' & \mbox{by Lemma \ref{general_properties} (\ref{cuasiConmutativeOfImplic})} \\
		& = & (((0 \to b) \to (c \to a)'') \to (c \to a)')'  & \mbox{by Lemma \ref{general_properties2} (\ref{281014_05})} \\
		& = & (((0 \to b) \to (c \to a)) \to (c \to a)')' & \mbox{by (\ref{eq_I20})} \\
		& = & (((c \to a)' \to (0 \to b)') \to (c \to a)')' & \mbox{by Lemma \ref{properties_of_I20_MC} (\ref{021017_01})} \\
		& = & ((((c \to a) \to 0) \to (0 \to b)') \to (c \to a)')' & \mbox{} \\
		& = & ((((c \to a) \to 0) \to (0 \to b)') \to (0' \to (c \to a))')' & \mbox{by Lemma \ref{general_properties_equiv} (\ref{TXX})} \\
		& = & ((0 \to b)' \to 0') \to (c \to a) & \mbox{using (I)} \\
		& = & (0 \to (0 \to b)) \to (c \to a) & \mbox{by Lemma \ref{general_properties} (\ref{cuasiConmutativeOfImplic})} \\
		& = & (0 \to b) \to (c \to a) & \mbox{by Lemma \ref{general_properties2} (\ref{031114_04})} \\
		& = & c \to ((0 \to b) \to a) & \mbox{by Lemma \ref{properties_of_I20_MC} (\ref{310516_01})} \\
		& = & c \to ((a \to b) \to a) & \mbox{by Lemma \ref{general_properties2} (\ref{291014_10})} \\
		& = & (a \to b) \to (c \to a) & \mbox{by Lemma \ref{properties_of_I20_MC} (\ref{310516_01})}. 
		\end{array}
		$$

		Hence
		\begin{equation} \label{230616_04}
		\mathcal{D}_{12} \subseteq \mathcal{D}_{35}.
		\end{equation}
		\ \\ \ \\ \ \\ \ \\
		Next, let $\mathbf A \in \mathcal{D}_{35}$ and $a,b,c \in A$. Then
				
		\noindent $a \to ((a \to b) \to c)$
		
		$
		\begin{array}{lcll}
		& = & (a \to b) \to (a \to c) & \mbox{by Lemma \ref{properties_of_I20_MC} (\ref{310516_01})} \\
		& = & [((a \to c)' \to a) \to (b \to (a \to c))')' & \mbox{by (I)} \\
		& = & ((((a \to c) \to 0) \to a) \to (b \to (a \to c))')' & \mbox{} \\
		& = & (((a \to c) \to (0 \to a)) \to (b \to (a \to c))')' & \mbox{by ($D_{3 5}$)} \\ 
		& = & ((0 \to ((a \to c) \to a)) \to (b \to (a \to c))')' & \mbox{by Lemma \ref{properties_of_I20_MC} (\ref{310516_01})} \\
		& = & ((0 \to ((0 \to c) \to a)) \to (b \to (a \to c))')' & \mbox{by Lemma \ref{general_properties2} (\ref{291014_10})} \\
		& = & (((0 \to c) \to (0 \to a)) \to (b \to (a \to c))')' & \mbox{by Lemma \ref{properties_of_I20_MC} (\ref{310516_01})} \\
		& = & ((0 \to (c \to a)) \to (b \to (a \to c))')' & \\
		&  & \mbox{by Lemma \ref{general_properties2} (\ref{071114_04}) and (\ref{311014_06})} &  \\
		& = & ((0 \to (a \to c)) \to (b \to (a \to c))')' & \mbox{by Lemma \ref{properties_A12} (\ref{140616_04})} \\
		& = & (((a \to c)' \to 0') \to (b \to (a \to c))')' & \mbox{by Lemma \ref{general_properties} (\ref{cuasiConmutativeOfImplic})} \\
		& = & (0' \to b) \to (a \to c) & \mbox{using (I)} \\
		& = & b \to (a \to c) & \mbox{by Lemma \ref{general_properties_equiv} (\ref{TXX})} \\
		& = & a \to (b \to c) & \mbox{by Lemma \ref{properties_of_I20_MC} (\ref{310516_01})} \\
		& = & a \to (a \to (b \to c)) & \mbox{by Lemma \ref{general_properties2} (\ref{031114_04})}. 
		\end{array}
		$
		
		As a consequence, we have that
		\begin{equation} \label{230616_06}
		\mathcal{D}_{35} \subseteq \mathcal{A}_{12}.
		\end{equation}
		
		Therefore, using (\ref{140616_05}), (\ref{230616_04}) and (\ref{230616_06}) we get that $$\mathcal{A}_{12} = \mathcal{D}_{12} = \mathcal{D}_{35}.$$
		
		\item 
		Let $\mathbf A \in \mathcal{A}_{35}$ and $a,b,c \in A$. Then
		$$
		\begin{array}{lcll}
		((a \to b) \to c) \to c & = & ((c' \to (a \to b)) \to (c \to c)')' & \mbox{by (I)} \\
		& = & ((c \to c) \to (c' \to (a \to b))')' & \mbox{by Lemma \ref{properties_of_I20_MC} (\ref{021017_02})} \\
		& = & ((c \to c) \to (c' \to (a \to b)))'' & \mbox{by ($A_{3 5}$)} \\ 
		& = & (c \to c) \to (c' \to (a \to b)) & {\mbox{by (\ref{eq_I20})} } \\
		& = & (c \to c) \to (a \to (c' \to b)) & \mbox{by Lemma \ref{properties_of_I20_MC} (\ref{310516_01})} \\
		& = & a \to ((c \to c) \to (c' \to b)) & \mbox{by Lemma \ref{properties_of_I20_MC} (\ref{310516_01})} \\
		& = & a \to ((c \to c) \to (c' \to b)'') & \mbox{by (\ref{eq_I20})} \\
		& = & a \to ((c \to c) \to (c' \to b)')' & \mbox{by ($A_{3 5}$)} \\ 
		& = & a \to ((c' \to b) \to (c \to c)')' & \mbox{by Lemma \ref{properties_of_I20_MC} (\ref{021017_02})} \\
		& = & a \to ((b \to c) \to c) & \mbox{by (I)}. 
		\end{array}
		$$
		Then $$\mathcal{A}_{35} \subseteq \mathcal{F}_{25}.$$
		For the converse, suppose $\mathbf A \in \mathcal{F}_{25}$ and $a,b,c \in A$.  Then we have that
		$$
		\begin{array}{lcll}
		0 \to (a \to a) & = & 0 \to ((0' \to a) \to a) & \mbox{by Lemma \ref{general_properties_equiv} (\ref{TXX})} \\
		& = & ((0 \to 0') \to a) \to a & \mbox{by ($F_{2 5}$)} \\ 
		& = & ((0'' \to 0') \to a) \to a & \mbox{by (\ref{eq_I20})} \\
		& = & a \to a & \mbox{by Lemma \ref{general_properties_equiv} (\ref{LeftImplicationwithtilde})}. 
		\end{array}
		$$
		Hence, 
		$$ 
		\mathbf A \models 0 \to (x \to x) \approx x \to x.
		$$
		Using Lemma \ref{040716_03} we have that $\mathbf A \in \mathcal{A}_{35}$.
		Thus,
		$$\mathcal{F}_{25} \subseteq \mathcal{A}_{35}.$$

\item 

Let $\mathbf A \in \mathcal{S}$ and $a,b,c \in A$. Then
$$
\begin{array}{lcll}
a \to ((b \to a) \to c) & = & (b \to a) \to (a \to c) & \mbox{by Lemma \ref{properties_of_I20_MC} (\ref{310516_01})} \\
& = & (((a \to c)' \to b) \to (a \to (a \to c))')' & \mbox{by (I)} \\
& = & (((a \to c)' \to b) \to (a \to c)')' & \mbox{by Lemma \ref{general_properties2} (\ref{031114_04})} \\
& = & ((b' \to (a \to c)) \to (a \to c)')' & \mbox{by Lemma \ref{properties_of_I20_MC} (\ref{310516_02})} \\
& = & ((b' \to (a'' \to c)) \to (a \to c)')'  & \mbox{by (\ref{eq_I20})} \\
& = & ((b' \to (c' \to a')) \to (a \to c)')' & \mbox{by Lemma \ref{properties_of_I20_MC} (\ref{310516_02})} \\
& = & ((c' \to (b' \to a')) \to (a \to c)')' & \mbox{by Lemma \ref{properties_of_I20_MC} (\ref{310516_01})} \\
& = & ((c' \to (a \to b)) \to (a \to c)')' & \mbox{by Lemma \ref{properties_of_I20_MC} (\ref{021017_01})}\\
& = & ((a \to b) \to a) \to c & \mbox{by (I)}. 
\end{array}
$$

Thus, $\mathcal{S} \subseteq \mathcal{B}_{25}$, implying that $ \mathcal{B}_{25} = \mathcal{S} $.  This completes the proof.
	\end{enumerate}

\end{proof}

Since $\mathcal{SL}$ coincides with some of Bol-Moufang varieties and $\mathcal{S}$ coincides with $\mathcal{B}_{25}$, we regard both of them as varieties of Bol-Moufang type, relative to $\mathcal{S}$.

We are now ready to present our main result of this paper.

\begin{Theorem} \label{theo_4_dist_var}
	There are 5 nontrivial varieties of Bol-Moufang type, relative to $\mathcal{S}$, that are distinct from each other:
	 $\mathcal{SL}$, $\mathcal A_{12}$,   $\mathcal A_{23}$,   $\mathcal F_{25}$ and $\mathcal S$.  These varieties satisfy the following inclusions: 
	\begin{enumerate}[(a)]
		\item $\mathcal{SL} \subset \mathcal A_{23} \subset \mathcal F_{25}$, 
		\item $\mathcal{SL} \subset \mathcal A_{12}$,
		\item $\mathcal{BA} \subset \mathcal A_{12} \subset \mathcal F_{25}$, 
		\item $\mathcal F_{25} \subset \mathcal{I}_{2,0} \cap \mathcal{MC}$,
		\item $\mathcal{SL} = \mathcal A_{23} \cap \mathcal A_{12}$.  
	\end{enumerate}	
\end{Theorem}

\begin{proof}
	\begin{enumerate}[(a)]
		\item By Lemma \ref{including_relations_of_SL} we know that $\mathcal{SL} \subseteq \mathcal A_{23}.$ 
		
		The following example shows that the inclusion is proper:\\

		\begin{tabular}{r|rrrr}
			$\to$: & 0 & 1 & 2 & 3\\
			\hline
			0 & 0 & 1 & 2 & 3 \\
			1 & 2 & 3 & 2 & 3 \\
			2 & 1 & 1 & 3 & 3 \\
			3 & 3 & 3 & 3 & 3
		\end{tabular}\\ 		
		Next, we wish to show that $\mathcal A_{23} \subseteq \mathcal F_{25}$.	
		So, let $\mathbf A \in \mathcal A_{23}$ and $a \in A$. Then 
		$$
		\begin{array}{lcll}
		0 \to (a \to a)	& = & (0 \to a) \to (0 \to a) & \mbox{by Lemma \ref{general_properties2} (\ref{071114_04}) and (\ref{311014_06})} \\
		& = & 0 \to ((0 \to a) \to a) & \mbox{by Lemma \ref{properties_of_I20_MC} (\ref{310516_01})} \\
		& = & (0 \to 0) \to (a \to a) & \mbox{by ($A_{2 3}$)} \\ 
		& = & a \to a & \mbox{by Lemma \ref{general_properties_equiv} (\ref{TXX})}. 
		\end{array}
		$$
		By Lemma \ref{040716_03} and Theorem \ref{theorem_040716_05} (\ref{040716_06}), we have
		$\mathcal A_{23} \subseteq \mathcal F_{25}$.
		
		The following example shows that the inclusion is proper: \\
		
		\begin{tabular}{r|rr}
			$\to$: & 0 & 1\\
			\hline
			0 & 1 & 1 \\
			1 & 0 & 1
		\end{tabular}\\
		
		\item By Lemma \ref{including_relations_of_SL} we know that $\mathcal{SL} \subseteq \mathcal A_{12}.$ \\
		The following example shows that the inclusion is proper:\\

		\begin{tabular}{r|rr}
			$\to$: & 0 & 1\\
			\hline
			0 & 1 & 1 \\
			1 & 0 & 1
		\end{tabular}\\

		\item By  Lemma \ref{including_relations_of_BA} we know that $\mathcal{BA} \subseteq \mathcal A_{12}.$ \\
		The following example shows that the inclusion is proper: \\

		\begin{tabular}{r|rr}
			$\to$: & 0 & 1\\
			\hline
			0 & 0 & 1 \\
			1 & 1 & 1
		\end{tabular}\\
		
		\ 		
		
		To prove $\mathcal A_{12} \subseteq \mathcal F_{25}$, we let
		$\mathbf A \in \mathcal A_{12}$ and $a \in A$. Then 
		$$
		\begin{array}{lcll}
		0 \to (a \to a) & = & (0 \to a) \to (0 \to a) & \mbox{by Lemma \ref{general_properties2} (\ref{071114_04}) and (\ref{311014_06})} \\
		& = & 0 \to (0 \to a) & \mbox{by Lemma \ref{properties_A12} (\ref{140616_01})} \\
		& = & 0 \to a & \mbox{by Lemma \ref{general_properties2} (\ref{031114_04})} \\
		& = & a \to a & \mbox{by Lemma \ref{properties_A12} (\ref{140616_01})}. 
		\end{array}
		$$
		By Lemma \ref{040716_03} and Theorem \ref{theorem_040716_05} (\ref{040716_06}), $\mathcal A_{12} \subseteq \mathcal F_{25}$.
		
		The following example shows that the inclusion is proper:\\

		\begin{tabular}{r|rrrr}
			$\to$: & 0 & 1 & 2 & 3\\
			\hline
			0 & 0 & 1 & 2 & 3 \\
			1 & 2 & 3 & 2 & 3 \\
			2 & 1 & 1 & 3 & 3 \\
			3 & 3 & 3 & 3 & 3
		\end{tabular}

		\item 
		
		The following example shows that the inclusion is proper:\\

		\begin{tabular}{r|rrr}
			$\to$: & 0 & 1 & 2\\
			\hline
			0 & 2 & 2 & 2 \\
			1 & 1 & 1 & 2 \\
			2 & 0 & 1 & 2
		\end{tabular}
		
		\item Let $\mathbf A \in \mathcal A_{23} \cap \mathcal A_{12}$ and $a \in A$. Then 
		
		$$
		\begin{array}{lcll}
		a	& = & a'' & \mbox{by (\ref{eq_I20})} \\
		& = & a' \to  0 & \mbox{} \\
		& = & 0' \to (a' \to 0) & \mbox{by Lemma \ref{general_properties_equiv} (\ref{TXX})} \\
		& = & (0 \to 0) \to (a' \to 0) & \mbox{} \\
		& = & 0 \to ((0 \to a') \to 0) & \mbox{by ($A_{2 3}$)} \\ 
		& = & 0 \to (0 \to (a' \to 0)) & \mbox{by ($A_{1 2}$)} \\ 
		& = & 0 \to (a' \to 0) & \mbox{by Lemma \ref{general_properties2} (\ref{031114_04})} \\
		& = & 0 \to a'' & \mbox{} \\
		& = & 0 \to a. & \mbox{by (\ref{eq_I20})} 
		\end{array}
		$$
		Hence 
		\begin{equation} \label{030816_01}
		\mathbf A \models x \approx 0 \to x.
		\end{equation}
		Consequently,
		\begin{equation} \label{030816_02}
		\mathbf A \models 0 \approx 0'.
		\end{equation}
		Therefore,
		$$
		\begin{array}{lcll}
		a \to a	& = & (a \to a)'' & \mbox{by (\ref{eq_I20})} \\
		& = & ((a \to a) \to 0)' & \mbox{} \\
		& = & ((a \to a) \to 0')' & \mbox{by (\ref{030816_02})} \\
		& = & ((a \to a) \to (0 \to 0))' & \mbox{} \\
		& = & (a \to ((a \to 0) \to 0))' & \mbox{by ($A_{2 3}$)} \\ 
		& = & (a \to (a \to (0 \to 0)))' & \mbox{by ($A_{1 2}$)} \\ 
		& = & (a \to (0 \to 0))' & \mbox{by Lemma \ref{general_properties2} (\ref{031114_04})} \\ 
		& = & (a \to 0')' & \mbox{} \\
		& = & (0 \to a')' & \mbox{by Lemma \ref{general_properties} (\ref{cuasiConmutativeOfImplic})} \\
		& = & (a')' & \mbox{by (\ref{030816_01})} \\
		& = & a. & \mbox{by (\ref{eq_I20})} 
		\end{array}
		$$
		Thus, $\mathbf A \models x \to x \approx x$.  Hence, from Lemma \ref{310516_09} and Lemma \ref{lemma_SL_I10_C}, $\mathbf A  \in \mathcal{SL}$. In view of Lemma \ref{including_relations_of_SL} the proof is complete.
	\end{enumerate}	
\end{proof}

From the preceding theorem, it is easy to see that the poset (in fact, $\land$-semilattice) of varieties of Bol-Moufang type, together with the variety $\mathcal{BA}$ of Boolean algebras, under inclusion is as shown below:\\

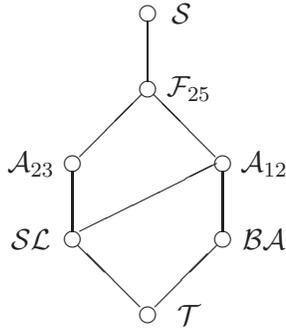
\begin{minipage}{0.5 \textwidth}
	\setlength{\unitlength}{1mm}
	\begin{picture}(30,45)(-15,0)
	\put(0,0){\circle{2}}
	\put(-10,10){\circle{2}}
	\put(10,10){\circle{2}}
	\put(-10,20){\circle{2}}
	\put(10,20){\circle{2}}
	\put(0,30){\circle{2}}
	\put(0,40){\circle{2}}
	\put(5.5,0){\makebox(0,0){$\mathcal T$}}
	\put(-15.5,10){\makebox(0,0){$\mathcal{SL}$}}
	\put(-15.5,20){\makebox(0,0){$\mathcal{A}_{23}$}}
	\put(15.5,10){\makebox(0,0){$\mathcal{BA}$}}
	\put(15.5,20){\makebox(0,0){$\mathcal{A}_{12}$}}
	\put(5.5,30){\makebox(0,0){$\mathcal{F}_{25}$}}
	\put(4.5,40){\makebox(0,0){$\mathcal{S}$}}                  
	\put(1,1){\line(1,1){8.2}}
	\put(-9,11){\line(2,1){18.2}}
	\put(-9,21){\line(1,1){8.2}}
	\put(-1,1){\line(-1,1){8.2}}
	\put(9,21){\line(-1,1){8.2}}
	\put(-10,11){\line(0,1){8}}
	\put(10,11){\line(0,1){8}}
	\put(0,31){\line(0,1){8}}
	\end{picture}
\end{minipage}\\

\medskip
\ \par

It is worthwhile to point out that our investigations into Bol-Moufang identities relative to $\mathcal{S}$, have revealed three new varieties, namely $\mathcal{A}_{12}$, $\mathcal{A}_{23}$ and $\mathcal{F}_{25}$, hitherto unknown.

In \cite{cornejo2016weak associative}, we investigate all the remaining weak associative laws of size $\leq 4$ which will complement the results of this paper. 

We conclude with the remark that it would be of interest to investigate the weak associative identities, and in particular, the identities of Bol-Moufang type, relative to $\mathcal{I}$ and other (important) subvarieties of $\mathcal{I}$
(see the Problem mentioned in the introduction). \\

\ \\ \ \\
\noindent{\bf Acknowledgment:} 
The first author wants to thank the
institutional support of CONICET  (Consejo Nacional de Investigaciones Cient\'ificas y T\'ecnicas).   
The authors wish to express their indebtedness to the anonymous referees for their careful reading of the paper.
\ \\

\noindent {\bf Compliance with Ethical Standards:}\\ 

\noindent {\bf Conflict of Interest:} The first author declares that he has no conflict of interest. The second author  declares that he has no conflict of interest. \\

\noindent {\bf Ethical approval:}
This article does not contain any studies with human participants or animals performed by any of the authors. \\

	\noindent {\bf Funding:}  
	The work of Juan M. Cornejo was supported by CONICET (Consejo Nacional de Investigaciones Cientificas y Tecnicas) and Universidad Nacional del Sur. 
	Hanamantagouda P. Sankappanavar did not receive any specific grant from funding agencies in the public, commercial, or not-for-profit sectors.

\vskip .5cm

\noindent {\sc Juan M. Cornejo}\\
Departamento de Matem\'atica\\
Universidad Nacional del Sur\\
Alem 1253, Bah\'ia Blanca, Argentina\\
INMABB - CONICET

\noindent jmcornejo@uns.edu.ar

\vskip .5cm

\noindent {\sc Hanamantagouda P. Sankappanavar}\\
Department of Mathematics\\
State University of New York\\
New Paltz, New York 12561\\
U.S.A.

\noindent sankapph@newpaltz.edu

\end{document}